\documentclass[11pt, reqno]{amsart}
\usepackage{amssymb, amstext, amscd, amsmath, framed}
\usepackage{color}

\usepackage{xy}
\xyoption{all}

\theoremstyle{plain}
\newtheorem{theorem}{Theorem}[section]
\newtheorem{thm}[theorem]{Theorem}
\newtheorem{cor}[theorem]{Corollary}
\newtheorem{prop}[theorem]{Proposition}
\newtheorem{lem}[theorem]{Lemma}
\newtheorem{obs}[theorem]{Observation}
\theoremstyle{definition}
\newtheorem{rem}[theorem]{Remark}

\newtheorem{defn}[theorem]{Definition}

\newtheorem{eg}[theorem]{Example}
\newtheorem{egs}[theorem]{Examples}

\newtheorem{ques}[theorem]{Question}
\newcommand{\bE}{{\mathbb{E}}}
\newcommand{\bF}{{\mathbb{F}}}

\newcommand{\bN}{{\mathbb{N}}}

\newcommand{\bZ}{{\mathbb{Z}}}

\newcommand{\bm}{{\mathbf{m}}}
\newcommand{\bn}{{\mathbf{n}}}

\newcommand{\bl}{{\mathbf{l}}}

\newcommand{\fk}{{\mathfrak{k}}}
\newcommand{\fs}{{\mathfrak{s}}}
\newcommand{\ft}{{\mathfrak{t}}}
\newcommand{\fu}{{\mathfrak{u}}}

\newcommand{\ep}{\varepsilon}
\renewcommand{\phi}{\varphi}
\newcommand{\upchi}{{\raise.35ex\hbox{\ensuremath{\chi}}}}

\newcommand{\qforal}{\quad\text{for all}\quad}

\newcommand{\dom}{\operatorname{dom}}
\newcommand{\codom}{\operatorname{codom}}

\newcommand{\id}{{\operatorname{id}}}

\newcommand{\im}{\operatorname{Im}}

\newcommand{\Fth}{\mathbb{F}_\theta^+}

\newcommand{\mt}{\varnothing}

\newcommand{\ol}{\overline}
\newcommand{\ul}{\underline}

\newcommand{\sym}{\rm Sym}

\begin{document}
\title[$k$-graphs $\&$ the YBE]
{The Interplay between $k$-graphs and \\ the Yang-Baxter Equation}
\author[D. Yang]
{Dilian Yang}
\address{Dilian Yang,
Department of Mathematics $\&$ Statistics, University of Windsor, Windsor, ON
N9B 3P4, CANADA} \email{dyang@uwindsor.ca}

\begin{abstract}
In this paper, we initiate the study of the interplay between $k$-graphs and the Yang-Baxter equation. 
For this, we provide two very different perspectives. 
One one hand, we show that the set of all set-theoretic solutions of the Yang-Baxter equation is a special class of 
single-vertex $k$-graphs. As a consequence, we construct an infinite family of large solutions of the 
Yang-Baxter equation from an arbitrarily given one. On the other hand, we prove that all single-vertex $k$-graphs are 
YB-semigroups of square-free, involutive solutions of the Yang-Baxter equation. 
Other various connections are also investigated. 

\end{abstract}

\thanks{The author was partially supported by an NSERC Discovery grant.}

\date{}
\maketitle

\section{Introduction}
\label{Intr}

The main aim of this paper is to initiate the study of the interplay between $k$-graphs and the Yang-Baxter equation,
both of which are very active research areas.  These two areas at first sight seem unrelated. 
However, the interplay between them turns out to be powerful, fascinating, and certainly deserves more attention.  
It should be mentioned that their link already appeared in several places in disguise. See, for example, \cite{FS01, RSY03, RS05, DY09}. 
Unfortunately, it seems that they have never been explored in detail according to my best knowledge. 
This paper can be thought of the algebraic theory of the interplay. The analytic theory will be studied elsewhere.

The Yang-Baxter equation was first introduced by C. N. Yang in the field of statistical mechanics \cite{Yan67}. It turns out to be also very 
useful in many other areas, such as quantum groups, link invariants, and the conformal field theory. Let $V$ be a vector space over a field $K$. A linear 
automorphism $R$ of $V\otimes V$ is called a \textit{solution of the Yang-Baxter equation} if 
\begin{align*}
r_{12}r_{23}r_{12}=r_{23}r_{12}r_{23}.
\end{align*} 
Here $r_{12}$ and $r_{23}$ are the leg notation: $r_{12}=r\otimes \id_V$ and $r_{23}=\id_V\otimes r$. It is generally a rather challenging task to find  all 
solutions of the Yang-Baxter equation. So V. G. Drinfeld suggested to find a class of solutions in \cite{Dri92},
which are known as set-theoretic solutions. Let $X$ be a non-empty set. A \textit{set-theoretic solution on $X$} is a bijection $R: X\times X\to X\times X$ satisfying 
\begin{align*}
R_{12}R_{23}R_{12}=R_{23}R_{12}R_{23}.
\end{align*} 
Obviously, every set-theoretic solution $R$ on $X$ determines a solution $r$ of the Yang-Baxter equation on the $K$-vector space spanned by $X$. 
Recently, there have been a lot of works, which deal with set-theoretic solutions. See, for example, \cite{CES04, ESS, GC, GM, GM2, GV98, LYZ, S} to name just a few, and the references therein. 

The study of $k$-graphs is relatively new. $k$-graphs were introduced by Kumjian and Pask in \cite{KP00}.
Their motivation was to generalize both directed graph C*-algebras (refer to \cite{Rae05} for the related references) 
and the Cuntz-Krieger algebras investigated by Robertson-Steger \cite{RS96, RS99},
via studying $k$-graph C*-algebras. 
A \textit{$k$-graph} is a countable small category $\Lambda$ with a degree map $d:\Lambda\to \bN^k$ satisfying a certain factorization property.
Due to the factorization property, $k$-graphs are much harder to study, compared with directed graphs. 
A special class of $k$-graphs, called single-vertex $k$-graphs, is singled out to study systematically in \cite{DPY08, DPY10, DY091, DY09, Pow07}. It turns out that 
this is a very interesting, nice and tractable class. Roughly speaking, a single-vertex $k$-graph is a unital cancellative semigroup, 
which are generated by $k$ types of generators $E_i:=\{e^i_1,\dots , e^i_{N_i}\}$ ($1\le i\le k$) satisfying some commutation relations 
determined by a set of permutations $\{\vartheta_{ij}, 1\le i<j\le k: \vartheta_{ij}\text{ is a permutation in } S_{N_i N_j}\}$ \cite{DY09}. 
The key observation/motivation for this paper is the following: to find all set-theoretic solutions of the Yang-Baxter equation is equivalent to study those 
single-vertex $k$-graphs with $k\ge 3$, all $N_i$ equal, and all $\vartheta_{ij}$ equal. 

The rest of this paper is structured as follows. In Section \ref{S:YBE}, the background on the Yang-Baxter equation is given. 
We review $k$-graphs in Section \ref{S:ps}. We also recall the equivalence between product systems of graphs over $\bN^k$ and 
$k$-graphs. The reason to do so is because it is easier to observe some important properties of $k$-graphs, which
will be used later. 
Our main section, Section \ref{S:relation}, explores the interplay between $k$-graphs and the Yang-Baxter equation, and show  
how they are beneficial to each other. The short section, Section \ref{S:pc}, describes some relations between product conjugacy and homomorphisms 
of solutions to the Yang-Baxter equation. In Section \ref{S:nondeg}, we prove the equivalence between the non-degeneracy of a solution to the Yang-Baxter equation and the unique pullback and pushout properties of its associated $k$-graph. An immediate consequence is that the non-degeneracy of a solution to the Yang-Baxter
equation implies the non-degeneracy 
of its higher levels. Finally, in the last section, we briefly discuss the homology and cohomology theories of the Yang-Baxter equation. 

\subsection*{Notation and Conventions}
Let us finish this section with the following notation and conventions. 

For a non-empty set $X$, we let $\bF_X^+$ denote the unital semigroup generated by $X$. 
For an integer $n\ge 1$, by $X^n$ we mean the Cartesian product of $n$ copies of $X$, and 
we often identify $(x_1,\ldots, x_n)\in X^n$ with the word $x_1\cdots x_n$ in $\bF_X^+$. 
The notation $\sym_X$ stands for the symmetric group of $X$. 
In this paper, for simplicity $X$ is assumed to be non-empty and finite, although most of our results also hold true even if $X$ is infinite. 

For a given map $R:X\times X\to X\times X$ and $n\ge 3$, we use $R_{ij}$ ($1\le i<j\le n$) to denote the leg notation of $R$ on $X^n$: 
$R_{ij}$ is equal to the map $R$ when restricted to the Cartesian product of the $i$th and $j$th positions,
and is equal to the identity map $\id_X$ when restricted to other positions. 
 
By $\bN$, we denote the set of all non-negative integers. 
We also use the following combinatorial notation:  for a natural number $N\ge 1$, let $[N]$ stand for the set $\{1,\ldots, N\}$.

\section{The Yang-Baxter Equation}
\label{S:YBE}

We present some necessary background on the Yang-Baxter equation in this section.
The material here is well-known in the literature. See, for example, \cite{CJdR10, ESS, GC, LYZ} for 
more information. 

\subsection{Definitions, Examples and Basic Properties}

Let $X$ be a non-empty set, and $R:X^2\to X^2$ be a bijection given by 
\begin{align}
\label{E:sigma0}
R(x,y)=(u,v).
\end{align} 

\begin{defn}
Let $R$ be defined in \eqref{E:sigma0}. 
We say that 
\begin{itemize}
\item
$R$ is \textit{involutive} if $R^2=\id_{X^2}$;
\item $R$ is \textit{square-free} if $R(x,x)=(x,x)$ for all $x\in X$;
\item $R$ is  \textit{non-degenerate} if the map $x\mapsto v$ is bijective for any fixed $y\in X$, and the map $y\mapsto u$ is bijective 
for any fixed $x\in X$. 
\end{itemize}
\end{defn}

\begin{defn}
\label{D:YBE}
Let $R:X^2\to X^2$ be a bijection. 

(i) We call $R$ a \textit{set-theoretical solution of the Yang-Baxter equation} (abbreviated as \textit{YBE}, or the \textit{braided relation})
if
\begin{align}
\label{E:YBE}
R_{12}R_{23}R_{12}=R_{23}R_{12}R_{23}
\end{align}
on $X^3$. We often simply call $R$ a \textit{YBE solution on $X$}. Sometimes, we write it as 
$R_X$\footnote{This may cause some confusion, since $\id_X$ is a YBE solution on $X$, while it as usual also denotes the identity map on $X$.
So, to avoid any possible confusion, we always write $\id_X$ to denote the identity map on $X$. 
} 
or a pair $(X,R)$ in order to emphasize it is a solution on $X$. 

(ii) An involutive, non-degenerate YBE solution $R$ on $X$ is often said to be \textit{symmetric}. 

(iii) Let $\sigma:X^2\to X^2$ be the flip map $\sigma(x,y)=(y,x)$ for all $x,y\in X$. If $R$ is a solution of YBE on $X$, then 
$r:=\sigma R$ is a \textit{set-theoretical solution of the quantum Yang-Baxter equation} (abbreviated as \textit{QYBE}) \textit{on $X$}:
\[
r_{12}r_{13}r_{23}=r_{23}r_{13}r_{12}.
\] 
As above, we often simply call $r$ a \textit{QYBE solution on $X$}.
\end{defn}

Let us rewrite $R$ in \eqref{E:YBE} as follows:
\begin{align}
\label{E:sigma}
R(x,y)=(\alpha_x(y), \beta_y(x))\qforal x, y\in X.
\end{align}
The characterizations given below in terms of $\alpha$ and $\beta$ are well-known and also easy to 
prove by simple calculations. 

\begin{lem}
\label{L:basic}
Let $R$ be a bijection on $X^2$ defined by \eqref{E:sigma}.  
Then we have the following: 
\begin{itemize}
\item[(i)] $R$ is square-free, if and only if for all $x\in X$
\[
\alpha_x(x)=\beta_x(x)=x \qforal x\in X. 
\]

\item[(ii)] $R$ is involutive, if and only if 
\[
\alpha_{\alpha_x(y)} \beta_y(x)=x\ \text{and}\ \beta_{\beta_y(x)} \alpha_x(y)=y\qforal x,y\in X.
\]

\item[(iii)] $R$ is non-degenerate, if and only if for any fixed $x\in X$, both $\alpha_x$ and $\beta_x$ are bijections on $X$.  

\item[(iv)] $R$ satisfies the YBE, if and only if for all $x,y,z\in X$
\begin{align}
\label{E:hom}
\alpha_x \alpha_y&=\alpha_{\alpha_x(y)} \alpha_{\beta_y(x)},\\
\label{E:antihom}
\beta_y \beta_x&=\beta_{\beta_y(x)} \beta_{\alpha_x(y)},\\
\label{E:comp}
\beta_{\alpha_{\beta_y(x)}(z)} \alpha_x(y)&=\alpha_{\beta_{\alpha_y(z)}(x)} \beta_z(y).
\end{align}
\end{itemize}
\end{lem}

It is sometimes convenient to write (\ref{E:hom}) and (\ref{E:antihom}) as 
\[
\alpha_x \alpha_y=\alpha_u \alpha_v\quad \text{and} \quad \beta_y\beta_x=\beta_v \beta_u
\]
respectively, provided that $R(x,y)=(u,v)$.

\begin{rem}
\label{R:lrderived}
By Lemma \ref{L:basic} (iv),  it is easy to see that $R(x,y)=(\beta_x(y),x)$ is a YBE solution on $X$, if and only if so is $R(x,y)=(y,\beta_y(x))$.  
Solutions of such forms are said to be \textit{of derived-type}. 
\end{rem}

\begin{eg}
\label{Eg:YB}
Some simple examples of YBE solutions on $[N]$ are given below.
\begin{enumerate}
\item
$R=\id$: $R(i,j)=(i,j)$ for all $i,j\in [N]$.  
\item
$R$ is the flip operator: $R(i,j)=(j,i)$ for all $i,j\in [N]$, where the addition is modulo $N$. This is called the \textit{trivial solution} on $[N]$. 
\item
$R(i,j)=(j+1,i+1)$ for all $i,j\in [N]$.
\item
$R(i,j)=(j+1,i)$ for all $i,j\in [N]$.
\item 
$R(i,j)=(j,{2j-i})$  for all $i,j\in [N]$, where $N\ge 3$ and the subtraction is modulo $N$. This is called the \textit{dihedral solution}  of the YBE on $[N]$. 
\end{enumerate}

When $N=2$, one can check that (i) - (iv) actually exhaust (up to isomorphism) all YBE solutions on $[2]$ (cf. \cite{Pow07} and Section \ref{S:pc} below). 
\end{eg}

\begin{eg}
\label{Eg:2.7}
Here is an easy generalization of Example \ref{Eg:YB}. 

(i) 
The YBE solutions $R$ in Example \ref{Eg:YB} (ii), (iii), (iv)
are of the form $R(x,y)=(f(y),g(x))$ for some mappings $f$ and $g$ on $X$. 
Notice that such a mapping $R$ is a YBE solution if and only if $fg=gf$.
Those solutions are called \textit{permutation solutions} of YBE on $X$.

(ii) 
A bijection $R(x,y)=(f(x),g(y))$ is a YBE solution on $X$ if and only if 
$f^2=f$, $g^2=g$ and $fgf=gfg$. 
But since $R$ is bijective, both $f$ and $g$ are invertible. So $f=g=\id_X$, implying $R=\id_{X^2}$. 
So Example \ref{Eg:YB} (i) is actually the only solution of such form. 
\end{eg}


\subsection{Some Known Constructions -- from Old to New}
\label{SS:oldnew}

In what follows, we review some known constructions of YBE solutions from given ones,
which will be useful later.

\subsection*{$1^\circ$ Cartesian Product Solutions \cite{ESS}}
Let 
\[
R_X(x_1,x_2)=(\alpha_{x_1}(x_2), \beta_{x_2}(x_1)) \text{ and } R_Y(y_1,y_2)=(\tilde\alpha_{y_1}(y_2), \tilde\beta_{y_2}(y_1))
\]
be YBE solutions on $X$ and $Y$, respectively. The
\textit{Cartesian product} of $R_X$ and $R_Y$ is a YBE solution $R$ on $X\times Y$ defined by 
\[
R\big((x_1,y_1),(x_2,y_2)\big)=\big((\alpha_{x_1}(x_2),\tilde\alpha_{y_1}(y_2)),(\beta_{x_2}(x_1), \tilde\beta_{y_2}(y_1))\big).
\]

\subsection*{$2^\circ$ Regular Extensions \cite{GM}} 
Let $R_X$ and $R_Y$ be YBE solutions on $X$ and $Y$, respectively. 
Motivated from \cite{ESS}, a very general construction of YBE solutions on the disjoint union $X\sqcup Y$,
called \textit{regular extensions of $R_X$ and $R_Y$},
is fully studied in \cite{GM}. For our purpose, here we only review two particular ones.

(i) \underline{The trivial extension}: It is the YBE solution $R$ on $X\sqcup Y$ defined via
\begin{align*}
R|_{X^2}&=R_X,\ R|_{Y^2}=R_Y,\ 
R|_{X\times Y}=\sigma_{X,Y},\ R|_{Y\times X}=\sigma_{Y,X}. 
\end{align*}
Here $\sigma_{X,Y}: X\times Y\to Y\times X$ is the flip operator; likewise for $\sigma_{Y,X}$.  

(ii) \underline{$\id_{X^2}\times_\theta\id_{Y^2}$}:
Let $R_X=\id_{X^2}$, $R_Y=\id_{Y^2}$, and $\theta: X\times Y\to Y\times X$ be a bijection. 
Define $R:(X\sqcup Y) \times (X\sqcup Y)\to (X\sqcup Y) \times (X\sqcup Y)$ to be 
\begin{align*}
R|_{X^2}=\id_{X^2}, \ R|_{Y^2}=\id_{Y^2},\ 
R|_{X\times Y}= \theta,\ R|_{Y\times X}=\theta^{-1}.
\end{align*}
Then $R$ is a YBE solution on $X\sqcup Y$, which is denoted as $\id_{X^2}\times_\theta \id_{Y^2}$.

\subsection*{$3^\circ$ Derived Solutions \cite{LYZ, S}}
Let $R$ be a non-degenerate YBE solution on $X$. Then from 
\begin{align}
\label{E:derived}
(x,y)\mapsto \big(\alpha_x(y),\beta_y(x)\big)\mapsto \big(\alpha_{\alpha_x(y)}\circ \beta_y(x), \beta_{\beta_y(x)}\circ \alpha_x(y)\big)
\end{align}
we get a new YBE solution $R'$ on $X$
\[
R'(\beta_y(x),y)=\big(\beta_{\beta_y(x)}\circ \alpha_x(y)), \beta_y(x)\big).
\]
Namely,
\[
R'(x,y)= \big(\beta_{x}\circ \alpha_{\beta_y^{-1}(x)}(y), x\big)\qforal x,y\in X.
\]
This solution $R'$ is called the \textit{derived solution} of $R$. 

Completely similar to above, it follows from \eqref{E:derived} that 
\[
(x,\alpha_x(y))\mapsto \big(\alpha_x(y),\alpha_{\alpha_x(y)}\circ\beta_y(x)\big)
\]
i.e.,
\[
(x,y)\mapsto \big(y,\alpha_y\circ\beta_{\alpha_x^{-1}(y)}(x)\big)
\]
is also a YBE solution on $X$ (\cite{LYZ}). But this one is closely related to $R'$ by Remark \ref{R:lrderived}.

One can easily see that if $R$ is symmetric, then $R'$ is the trivial solution. It is shown in \cite{ESS, LYZ, S} that $R$ and $R'$ are equivariant via a very special map.

\subsection{YB-Semigroups and YB-Groups}

Let us associate to a given YBE solution two important objects -- its YB-semigroup and YB-group. 
\begin{defn}
Let $R$ be a YBE solution on $X$.

$\blacktriangleright$ The \textit{YB-semigroup of $R$}, denoted as $G_R^+$, is the unital semigroup generated by $X$ with commutation relations determined by $R$:
\[
G_{R}^+={}_{\text{sgp}}\big\langle X; xy=y'x' \text{ whenver }R(x,y)=(y',x')\big\rangle.
\]

$\blacktriangleright$ The \textit{YB-group of $R$}, denoted as $G_R$, is the group generated by $X$ with commutation relations determined by $R$:
\[
G_{R}={}_{\text{gp}}\big\langle X; xy=y'x' \text{ whenever }R(x,y)=(y',x')\big\rangle.
\]
\end{defn}

We should mention that, in the literature on the YBE,\footnote{There is a bias on the notation here: in the literature on the YBE, $G_R^+$ and $G_R$ are often written as $G_X^+$ and $G_X$, 
respectively. Our intention here is to unify the notation for the YBE and the notation used for $k$-graphs later.}  
$G_R^+$ and $G_R$ are also known as the \textit{structure semigroup} and \textit{structure group} of $R$ on $X$,
because they play vital roles in the classification of YBE solutions \cite{ESS, GM, GV98, LYZ, S}. In particular, symmetric YBE solutions are completely classified 
by bijective 1-cocycle quadruples in \cite{ESS}; it is shown in \cite{GV98} that they are essentially the same as semigroups of I-type, and closely
related to semigroups of skew polynomial type, Bieberbach groups.
YBE solutions on groups are completely characterized in \cite{LYZ}. 
Also, some results in \cite{ESS} are generalized to arbitrary non-degenerate YBE solutions in \cite{S}.


One can easily rephrase the characterization given in Lemma \ref{L:basic} (iv) in terms of representations of YB-semigroups and YB-groups
(cf., e.g., \cite{ESS, GM}). 

\begin{cor}
\label{C:char}
A map $R(x,y)=(\alpha_x(y), \beta_y(x))$ is a YBE solution on $X$, if and only if
\begin{itemize}
\item[(i)] $\alpha$ can be extended to a left action of $G_R$ and/or $G_R^+$ on $X$, 
\item[(ii)] $\beta$ can be extended to a right action of $G_R$ and/or $G_R^+$ on $X$, and 
\item[(iii)] the compatibility condition \eqref{E:comp} holds. 
\end{itemize}
 \end{cor}

\begin{egs}

Let us consider the following examples.

(i) If $R=\id$ on $[N]$, then $G_R^+=\bF_N^+$ and $G_R=\bF_N$.

(ii) If $R$ is the trivial YBE solution on $[N]$, then $G_R^+=\bN^N$ and $G_R=\bZ^N$. 

(iii) If $R$ is the dihedral solution on $[3]$, then 
\begin{align*}
G_R^+
&=\langle e_1, e_2,e_3: e_1e_2=e_2e_3=e_3e_1, \ e_2e_1=e_3e_2=e_1e_3\rangle,\\
G_R
&=\langle e_1, e_2: e_1^2=e_2^2, \ e_1e_2e_1=e_2e_1e_2\rangle.
\end{align*}
Let $e_\mt$ be the identity of $G_R$. Then $G_R/\langle e_1^2-e_\mt, e_2^2-e_\mt\rangle$ is isomorphic to the dihedral group $D_3$.

\end{egs}

\section{$k$-graphs}

\label{S:ps}

$k$-graphs have been attracting a great deal of attention because they provide a lot of interesting operator algebras 
and have some nice applications.
In this section, we only give a very short account. 
The main sources of this sections are \cite{FS01, KP00, RS05, DY09}.

\subsection{$k$-Graphs}

Let $k \ge 1$ be an integer  and regard the monoid $\bN^k$ as a category with one object. Also we write $\{\epsilon_i:1\le i\le k\}$ to 
denote the standard basis of $\bN^k$.  

\begin{defn}
\label{D:Pgr}
A \textit{$k$-graph} (or \textit{higher-rank graph}) is a countable small category $\Lambda$ with a degree map $d:\Lambda\to \bN^k$ satisfying the following 
factorization property:
\begin{center}
whenever $\lambda\in\Lambda$ satisfies $d(\lambda)=\bm+\bn$, there are unique elements $\mu, \nu\in\Lambda$ such that 
$d(\mu)=\bm$, $d(\nu)=\bn$ and $\lambda=\mu\nu$. 
\end{center}
\end{defn}

Let us recall that a (directed) graph $E=(E^0,E^1,r_E,s_E)$ consists of non-empty countable sets $E^0$ (called the vertex set), $E^1$ (called
the edge set), the source and range maps $s_E,r_E:E^1\to E^0$. 
Let $E^n=\{\lambda=e_1e_2\cdots e_n: e_1,\ldots, e_n\in E^1, r_E(e_i)=s_E(e_{i+1}) \text{ for }1\le i\le n-1\}$.
One can extend the source and range maps to $E^n$: $s_E(\lambda)=s_E(e_1)$ and $r_E(\lambda)=r_E(e_n)$ 
for $\lambda=e_1\cdots e_n\in E^n$. The path space of $E$ is defined to be
$E^*=\bigcup_{n=0}^\infty E^n$.
A graph can be naturally identified with $1$-graphs as follows. Let $\Lambda=E^*$, $d(\lambda)=n$,
$r(\lambda)=s_E(\lambda)$ and $s(\lambda)=r_E(\lambda)$  for all $\lambda\in E^n$. 
Thus $k$-graphs are a higher dimensional generalization of directed graphs. 

For $\bn\in \bN^k$, let $\Lambda^\bn=d^{-1}(\bn)$. Naturally, $\Lambda^{\mathbf{0}}$ and $\cup_{i=1}^k\Lambda^{\epsilon_i}$ can be 
regarded as the vertex and edge sets of $\Lambda$, respectively. It is often useful to think of a $k$-graph as a graph whose edges are labelled in $k$ 
(distinct) colours. 
There are source and range maps $s, r: \Lambda\to \Lambda^{\mathbf{0}}$ such that $r(\lambda)\lambda s(\lambda)=\lambda$ for all $\lambda\in\Lambda$.
For $v\in\Lambda^{\mathbf{0}}$ and $\bn\in\bN^k$,
set $v\Lambda^\bn=\{\lambda\in\Lambda: r(\lambda)=v, d(\lambda)=\bn\}$. 
We say that a $k$-graph $\Lambda$ has \textit{no sources} (resp. be \textit{row-finite})  if 
$|v\Lambda^\bn|>0$ (resp. $|v\Lambda^\bn|<\infty$)
for all $v\in\Lambda^{\mathbf{0}}$ and $\bn\in \bN^k$.
In this paper, all $k$-graphs are always assumed to have no sources and be row-finite. 

It is proved in \cite{FS01, RS05} that the class of $k$-graphs is equivalent to the class of product systems of graphs over $\bN^k$. 
Since we will use this relation later, let us digress for a while to recall the notion of such product systems. 

Let $E=(E^0, E^1, r_E, s_E)$ and  $F=(E^0, F^1, r_F, s_F)$ be two graphs with the same 
vertex set $E^0$.  Construct a new directed graph $E\times_{E^0} F$, where
\begin{align*} 
&(E\times_{E^0}F)^0=E^0,\\ 
&(E\times_{E^0}F)^1=\big\{(\lambda,\mu):\lambda\in E^1, \mu\in F^1, r_E(\lambda)=s_F(\mu)\big\},\\
&s(\lambda,\mu)=s_E(\lambda),\  r(\lambda,\mu)=r_F(\mu).
\end{align*}

\begin{defn}
\label{D:ps}
A \textit{product system $(E,\alpha)$ of graphs over $\bN^k$} consists of a family of graphs $\big\{E_\bn:=(E^0, E_\bn^1, r_\bn,s_\bn):\bn\in \bN^k\big\}$,
which  have common vertex set $E^0$ and disjoint edge sets $E_\bn^1$ ($\bn\in \bN^k$), such that the following properties hold: 
\begin{itemize}
\item for all $\bm,\bn\in \bN^k$, there are graph isomorphisms $\alpha_{\bm,\bn}:E_\bm\times_{E^0}E_\bn\to E_{\bm+\bn}$ satisfying the associativity condition
\[
\alpha_{\bm+\bn,\bl}(\alpha_{\bm,\bn}(\lambda,\mu),\nu)=\alpha_{\bm,\bn+\bl}(\lambda,\alpha_{\bn,\bl}(\mu,\nu))\quad (\bl, \bm, \bn\in\bN^k)
\]
for all $(\lambda,\mu)\in(E_\bm\times_{E^0}E_\bn)^1$ and $(\mu,\nu)\in(E_\bn\times_{E^0}E_\bl)^1$;

\item
$
E_{\mathbf 0}=(E^0,E^0,\id_{E^0},\id_{E^0}).
$
\end{itemize}
\end{defn}


Here is the equivalence between $k$-graphs and product systems of graphs over $\bN^k$ (\cite{FS01, RS05}). 
For a given product system $(E,\alpha)$ of graphs over $\bN^k$,
let $\Lambda_E$ be the category with objects $E^0$ and morphisms $E^1:=\bigsqcup_{\bn\in \bN^k}E_\bn^1$, with $s(\lambda)=\dom(\lambda)$, 
$r(\lambda)=\codom(\lambda)$, $d(\lambda)=\bn$ for all $\lambda\in E_\bn^1$.
Then $(\Lambda_E, d)$ is a $k$-graph. 

Conversely, for a given $k$-graph $(\Lambda,d)$,  the factorization property determines an isomorphism 
$\alpha_{\bm, \bn}: \Lambda^\bm\times_{\Lambda^{\mathbf{0}}} \Lambda^\bn\to \Lambda^{\bm+\bn}$ for all $\bm, \bn\in\bN^k$. 
Now let $(E_\Lambda)^0:=\Lambda^{\mathbf{0}}$, 
$(E_\Lambda)_\bn^1:=\Lambda^\bn=\{\lambda\in\Lambda: d(\lambda)=\bn\}$ 
for all $\bn\in \bN^k$, $\dom(\lambda)=s(\lambda)$, and $\codom(\lambda)=r(\lambda)$ for all $\lambda\in\Lambda$. 
Then $(E,\alpha)$ is a product system of graphs over $\bN^k$.

The following simple but very useful observation becomes obvious now.

\begin{obs}
\label{O:subps}
Let $\Lambda$ be a $k$-graph, and $S$ be a sub-semigroup of $\bN^k$ including $\mathbf{0}$, which has rank $n\in\bN$. 
Then the restriction of $\Lambda$ onto $S$:
\[
\Lambda_S:=\big\{\lambda\in\Lambda: d(\lambda)\in S\big\},
\]
is an $n$-graph. 
\end{obs}

\medskip 

Let $\Lambda$ be a given $k$-graph.  As mentioned above, there is an isomorphism $\alpha_{\epsilon_i,\epsilon_j}$ 
from $\Lambda^{\epsilon_i}\times_{\Lambda^{\mathbf{0}}} \Lambda^{\epsilon_j}$ onto $\Lambda^{\epsilon_i+\epsilon_j}$ for all $1\le i<j\le k$. Thus 
one obtains the following isomorphism:
\begin{align*}
\theta_{ij}=\alpha_{\epsilon_j,\epsilon_i}^{-1}\circ\alpha_{\epsilon_i,\epsilon_j}:
 \Lambda^{\epsilon_i}\times_{\Lambda^{\mathbf{0}}} \Lambda^{\epsilon_j}\to \Lambda^{\epsilon_j}\times_{\Lambda^{\mathbf{0}}} \Lambda^{\epsilon_i}\quad (1\le i<j\le k).
\end{align*}
A natural question is if a given family $\{\theta_{ij}: 1\le i<j\le k\}$ of isomorphisms of countable graphs determines a $k$-graph. 
This is the case for $k=2$ (\cite{KP00}).  However, this fails for $k\ge 3$.
It is shown in \cite{FS01} that this is the case for $k\ge 3$, if and only if for  $1 \le i< j < \fk \le k$ one has 
\begin{align}
\label{E:kQYB}
(\id_\fk\times \theta_{ij})(\theta_{i\fk}\times \id_j)(\id_i\times \theta_{j\fk})=(\theta_{j\fk}\times \id_i)(\id_j\times \theta_{i\fk})(\theta_{ij}\times \id_\fk),
\end{align}
as isomorphisms from $\Lambda^{\epsilon_i}\times_{\Lambda^{\mathbf{0}}} \Lambda^{\epsilon_j}\times_{\Lambda^{\mathbf{0}}} \Lambda^{\epsilon_\fk}$ to 
$\Lambda^{\epsilon_\fk}\times_{\Lambda^{\mathbf{0}}} \Lambda^{\epsilon_j}\times_{\Lambda^{\mathbf{0}}} \Lambda^{\epsilon_i}$. 
Here $\id_i$ is the identity on $\Lambda^{\epsilon_i}$.
This result is a key to studying the interplay between $k$-graphs and the YBE, and provides the starting point of this paper.

\subsection{Single-vertex $k$-graphs}
\label{SS:single}

Roughly speaking, a single vertex $k$-graph $\Lambda$ is a unital cancellative semigroup of special form \cite{DY09}.
To be more precise, for $1\le i\le k$, let 
\[
E_i:=\{e^i_1,\dots , e^i_{N_i}\} 
\]
be the set of all edges in $\Lambda$ of degree $\ep_i$. 
It follows from the factorization property of $\Lambda$ that,
for $1\le i<j\le k$, there is a bijection $\theta_{ij}: [N_i]\times[N_j]\to [N_j]\times [N_i]$
such that the following \textit{$\theta$-commutation relations} hold:
\begin{align*}
\tag{$\theta$-CR}
 e^i_\fs e^j_\ft = e^j_{\ft'} e^i_{\fs'} 
 \quad\text{if}\quad 
 \theta_{ij}(\fs,\ft) = (\ft',\fs') .
\end{align*}
Then $\Lambda$ coincides with the unital semigroup $\Fth$ defined as follows 
\cite{DY09}\footnote{Note the difference on the notation: the map $\theta_{ij}$ in \cite{DY09} is actually the same as $\vartheta_{ij}$ used 
\eqref{E:kYB} below.}
\begin{align*}
\Fth=\big\langle e_\ft^i\in E_i,\, \ft\in [N_i],\, 1\le i\le k;\, \text{($\theta$-CR)}\big\rangle,
\end{align*}
which is occasionally also written as
\begin{align}
\label{E:Fth}
\Fth=\big\langle e_\ft^i\in E_i,\, 1\le i\le k; \, \theta_{ij}, \, 1\le i<j\le k \big\rangle.
\end{align}
It is worthwhile to mention that $\Fth$ has (left and right) cancellation property due to the factorization property of $\Lambda$. 
It follows from ($\theta$-CR) that every element $x\in \Fth$ can be uniquely written as
$
x=e_{u_1}^1\cdots e_{u_k}^k
$
for some $u_i\in\bF_{N_i}^+$ ($1\le i\le k$). Here we use 
the multi-index notation: $e^i_{u_i}=e^i_{\fs_1}\cdots e^i_{\fs_n}$ if $u_i=\fs_1\cdots\fs_n\in\bF_{N_i}^+$.
The degree map $d:\Fth\to \bN^k$ is given by $d(x)=(|u_1|,\ldots, |u_k|)$,
where $|u_i|$ is the length of $u_i$. 
Sometimes, in order to emphasize that $\Fth$ is a $k$-graph (not just a usual semigroup), 
we write it as a pair $(\Fth,d)$. 

Let $\sigma_{j,i}: [N_j]\times [N_i]\to [N_i]\times [N_j]$ be the flip operator, and $\vartheta_{ij}=\sigma_{j,i}\circ\theta_{ij}$. 
Then $\vartheta_{ij}$ is a bijection on $[N_i]\times [N_j]$.
For convenience, we identify each bijection $\vartheta_{ij}$ as a bijection on $\prod_{i=1}^k [N_i]$, which fixes the coordinates except for
$i,j$, on which it acts as $\vartheta_{ij}$. Abusing the notation, we still denote the corresponding bijection on  $\prod_{i=1}^k [N_i]$ as $\vartheta_{ij}$. 
 Then one can rephrase the identity \eqref{E:kQYB} as follows: the family $\vartheta=\{\vartheta_{ij}:1\le i<j\le k\}$ satisfies the following generalized QYBE:
\begin{align}
\label{E:kYB}
 \vartheta_{ij}\vartheta_{i\fk}\vartheta_{j\fk}=\vartheta_{j\fk}\vartheta_{i\fk}\vartheta_{ij}
 \qforal 1 \le i< j <\fk \le k.
\end{align}
It follows from \eqref{E:kQYB} that Eq.~\eqref{E:kYB} exactly characterizes single-vertex $k$-graphs.  
As observed in \cite{DY09}, using Observation \ref{O:subps} one has the following. 

\begin{obs}
\label{O:3-k}
Let $k \ge 3$ and $\theta_{ij}: [N_i]\times [N_j]\to [N_j]\times [N_i]$ be bijections for all $1\le i<j\le k$. Then $\{\theta_{ij}:1\le i<j\le k\}$ 
determines a single-vertex $k$-graph $\Fth$, if and only if the restriction of $\Fth$
to every triple family $\{ e^i_\fs,\ e^j_\ft,\ e^{\fk}_\fu \}$ determines a (single-vertex) 3-graph.
\end{obs}

\section{Between $k$-graphs and the YBE: back and forth}

\label{S:relation}

This long section, which is the main section of this paper, explores the interplay between single-vertex $k$-graphs and the YBE. 
For this, we provide two very different perspectives which seem contradictory in appearance.
On one hand, we show that to obtain solutions of the YBE is equivalent to study 
a very special class of single-vertex $k$-graphs with $k\ge 3$ (Theorem \ref{T:YBkGr}); 
on the other hand, we prove that all single-vertex $k$-graphs ($k\ge 3$) can be obtained 
from a special class of involutive, square-free solutions of the YBE (Theorem \ref{T:equ}).  

Here are two important consequences. Firstly, we exhibit a family of ``higher level" YBE solutions $\{R^{n,n}: n\ge 1\}$ from 
an arbitrarily given one $R$ (Theorem \ref{T:cons}). The sizes of those solutions can be as large as one wishes.
Moreover, these new solutions $R^{n,n}$ are involutive (resp. of derived-type), if and only if $R$ has respective properties.
This substantially generalizes the main result of \cite{BC14} (also cf. \cite{CJdR10}). 
Secondly, all single-vertex $k$-graphs ($k\ge 3$) are nothing but special cancellative YB-semigroups.  
We believe that this will be very helpful in the future studies on the relations of between operator algebras associated to $k$-graphs 
and algebras associated to YB-semigroups/groups. 

Besides above, the periodicity of a YBE solutions in terms of its associated $k$-graphs is investigated.  
Also, we discuss connections with some known results in \cite{GM, LYZ} from our perspectives. 

\medskip

Our first result in this section follows from Eq.~\eqref{E:kYB} and Observation \ref{O:3-k}.

\begin{thm}[and \textbf{Definition}]
\label{T:YBkGr}
Let $N\ge 1$ be an integer and $R$ be a bijection on $[N]^2$. Then the following statements are equivalent:
\begin{itemize}
 \item[(i)] 
  $R$ is a solution of the YBE on $[N]$;
  
  \item[(ii)] 
 $R$ determines a single-vertex $3$-graph $\bF_R^+$:
 \begin{align*}
 \bF_R^+=
 \left\langle
e_\ft^i,\, \ft\in[N],\, 1\le i\le 3;
\begin{array}{l}
 e_{\fs}^ie_{\ft}^j=e_{\ft'}^j e_{\fs'}^i \text{ for }1\le i<j\le 3\\
 \text{if }  R(\fs,\ft)=(\ft',\fs'),\ \fs,\ft,\fs',\ft'\in[N]
\end{array}
\right\rangle.
\end{align*}

 \item[(iii)] 
 For every integer $k\ge 3$,
$R$ determines a single-vertex $k$-graph $\bF_R^+$: 
\begin{align*}
 \bF_R^+=\left\langle e_\ft^i,\, \ft\in[N],\, 1\le i\le k;
 \begin{array}{l}
 e_{\fs}^ie_{\ft}^j=e_{\ft'}^j e_{\fs'}^i\text{ for }1\le i<j\le k\\
\text{if }  R(\fs,\ft)=(\ft',\fs'),\ \fs,\ft,\fs',\ft'\in[N]
\end{array}\right\rangle.
 \end{align*}
 \end{itemize}
We call $\bF_R^+$ in (iii) the {\rm associated $k$-graph of $R$}. 
\end{thm}

\begin{proof}
(i) $\Leftrightarrow$ (ii): 
From Subsection \ref{SS:single}, 
$R$ is a solution of the YBE on $[N]$, if and only if 
$\sigma R$ satisfies the QYBE on $[N]$ (cf. Definition \ref{D:YBE} (iii)), where $\sigma$ is the flip operator on $[N]\times [N]$,
if and only if $\vartheta=\{\vartheta_{12}=\sigma R, \vartheta_{13}=\sigma R, \vartheta_{23}=\sigma R\}$ satisfies 
satisfies Eq.~\eqref{E:kYB} on $[N]^3$, 
if and only if $\bF_R^+$ given in (ii) is a $3$-graph.  

(ii) $\Leftrightarrow$ (iii): This follows from Observation \ref{O:3-k}. 
\end{proof}

Note that $N$ above  can be infinite. 

\begin{rem}
\label{R:2YB}

(i) Note that $R$ is independent of the `colours' of its associated $k$-graph.  

(ii) Let $X$ be an arbitrary set indexed by $[N]$, say $X=\{e_\fs:\fs\in[N]\}$. Let $\bar R:X^2\to X^2$ be defined by
\[
\bar R(e_\fs,e_\ft)=(e_{\ft'},e_{\fs'})\quad \text{whenever}\quad R(\fs,\ft)=(\ft',\fs').
\]
Obviously, $R$ satisfies the YBE on $[N]$ if and only if $\bar R$ satisfies the YBE on $X$. So we
often identity $R$ with $\bar R$. 

(iii) As seen from Section \ref{S:ps}, any bijection (not necessarily satisfying the YBE) $R$ on $X^2$ determines a (single-vertex) $2$-graph.  
\end{rem}

\begin{rem}
Construction $1^\circ$ and Construction $2^\circ$ in Section \ref{S:YBE} are special cases of the multiplication of the second 
cohomological group in \cite{FS01} and the Cartesian product of $k$-graphs in \cite{KP00}, respectively, 
while Construction $3^\circ$ there gives a new construction of $k$-graphs. 
\end{rem}

Due to Theorem \ref{T:YBkGr}, in order to study the relations between the YBE and $k$-graphs,
without loss of generality we will assume in \eqref{E:Fth} that 
\begin{align}
\label{E:ass}
\framebox[1.1\width]{$k=3,\, N_1=N_2=N_3=:N,\, \theta_{12}=\theta_{13}=\theta_{23}=:R$}
\end{align}
in the rest of the paper, unless otherwise specified. 

\medskip
In the subsection below, we will construct  an infinite family of `large' YBE solutions.

\subsection{A New Family of YBE Solutions}
\label{SS:newfamily}

Let $R:[N]^2\to [N]^2$ be an arbitrary bijection, and $E_i$ be a set indexed by $[N]$, say 
\[
E_i=\{e^i_\fs:\fs\in [N]\} \quad (1\le i\le 3).
\] 
Although $R$ is not necessarily a solution of the YBE, 
one can still construct a unital semigroup $\bF_R^+$ as in Theorem \ref{T:YBkGr} (ii). 

Let $u$ and $v$ be two words in $\bF_N^+$ with $|u|=l$ and $|v|=m$. By the commutation relations one has
\[
e_u^i e^j_v=e^j_{v'} e^i_{u'}\quad (1\le i<j\le 3).
\]
Thus the bijection $R$ on $[N]^2$ induces two bijections on `higher levels':
\begin{align*}
 R^{l, m}:\, &[N^l]\times [N^{m}]\to [N^{m}]\times [N^{l}],\quad R^{l, m}(u,v)=(v',u'),\\
 \ R_{ij}^{l, m}: \,&E_i^l\times E_j^{m}\to E_j^{m}\times E_i^l,\quad R_{ij}^{l, m}(e^i_u,e^j_v)=(e^j_{v'},e_{u'}^i).
\end{align*}
Here $E_i^l=\{\mu\in\Fth: d(\mu)=l\ep_i\}$ for $1\le i\le k$ and $l\ge 1$. 
Obviously, $R^{1,1}=R$.

Using the relations between product systems over graphs and $k$-graphs, 
one sees from Observation \ref{O:subps} that 
$R$ determines a single-vertex $3$-graph, if and only if, for all positive integers $l, m,n$, 
the triple $\{R^{l,m}, R^{l,n},R^{m,n}\}$ determines a 
single-vertex $3$-graph by letting $S=\langle l\epsilon_1, m\epsilon_2, n\epsilon_3\rangle$. 
More precisely, we have the following. 

\begin{obs}
\label{O:ext}
Keep the same notation as above. 
 For integers $l,m,n\ge 1$, let 
\begin{align*}
\bE_1&:=E_1^l, \ \bE_2:=E_2^m,\  \bE_3:=E_3^n,\\
\theta_{12}&=R^{l,m}, \ \theta_{13}=R^{l,n}, \ \theta_{23}=R^{m,n},\\
\vartheta_{12}&=\sigma_{2,1}\theta_{12}, \ \vartheta_{13}=\sigma_{3,1}\theta_{13}, \ \vartheta_{23}=\sigma_{3,2}\theta_{23}.
\end{align*}
Then $\bF_R^+$ is a single vertex $3$-graph, 
if and only if  $\{\vartheta_{12},\,  \vartheta_{13},\, \vartheta_{23}\}$ satisfies Eq.~\eqref{E:kYB}, 
if and only if
 $\bF_\theta^+=\big\langle\bE_1,\bE_2,\bE_3; \theta_{12}, \theta_{13}, \theta_{23}\big\rangle$ 
 (refer to \eqref{E:Fth} for notation) is a single vertex $3$-graph. 
\end{obs}

Notice that it is not required that $\bF_\theta^+$ in Observation \ref{O:ext} satisfy condition \eqref{E:ass}. 
We are now ready to give our new construction as promised, which describes an infinite family of big YBE solutions from a given one,
and generalizes the main result of \cite{BC14} (cf. also \cite{CJdR10}).

\begin{thm}[\textsf{A New Construction}]
\label{T:cons}
A bijection $R:[N]^2\to [N]^2$ is a YBE solution, if and only if  its $n$th level $R^{n,n}$ is a YBE solution on $[N^n]$ for any integer $n\ge 1$. 

Furthermore, the following holds true:
\begin{itemize}
\item[(i)] if $R$ is involutive, then so is $R^{n,n}$ for every $n\ge 1$;

\item[(ii)] if $R$ is of derived type, then so is $R^{n,n}$ for every $n\ge 1$. 
\end{itemize}
\end{thm}

\begin{proof}
The first part is immediate from Theorem \ref{T:YBkGr} and Observation \ref{O:ext}.
We now show that they have the required properties. 

(i) Suppose that $R$ is involutive. 
Identify $(x_1\cdots x_n, y_1\cdots y_n)\in [N^n]^2$ with the element $(x_1,\ldots, x_n, y_1,\ldots, y_n)$ in $[N]^{2n}$. 
Then
\[
R^{n,n}
=(R_{n\, n+1}\cdots R_{12})(R_{n+1\, n+2}\cdots R_{23})\cdots (R_{2n-1\, 2n}\cdots R_{n\, n+1}),
\]
where $R_{ij}$'s are the leg notations of $R$ on $[N]^{2n}$. 
It is apparent that 
\begin{align}
\label{E:comm}
R_{m,n}R_{ij}=R_{ij}R_{m,n}\text{ for all }m<n,\, i<j \text{ with } [m,n]\cap[i,j]=\mt.
\end{align}
Then we obtain that 
\begin{align*}
(R^{n,n})^2
=&(R_{n\, n+1}\cdots R_{12})\cdots (R_{2n-1\, 2n}\cdots R_{n\, n+1})\cdot\\ 
& (R_{n\, n+1}\cdots R_{12})\cdots (R_{2n-1\, 2n}\cdots R_{n\, n+1})\\
=&(R_{n\, n+1}\cdots R_{12})\cdots (R_{2n-1\, 2n}\cdots R_{n+1\, n+2})\cdot R_{n\, n+1}R_{n\, n+1}\cdot \\
&(R_{n-1\, n}\cdots R_{12})\cdots (R_{2n-1\, 2n}\cdots R_{n\, n+1})\\
=&(R_{n\, n+1}\cdots R_{12})\cdots (R_{2n-2\, 2n-1}\cdots R_{n-1\, n}) (R_{2n-1\, 2n}\cdots R_{n+1\, n+2})\cdot \\
&(R_{n-1\, n}\cdots R_{12})\cdots (R_{2n-1\, 2n}\cdots R_{n\, n+1})\quad (\text{as }R_{n\, n+1}^2=\id).
\end{align*}
From \eqref{E:comm}, one has that $R_{n-1\, n}$ and $R_{2n-1\, 2n}\cdots R_{n+1\, n+2}$ commute. So we switch their positions to yield
\begin{align*}
(R^{n,n})^2
\stackrel{(a)}=&(R_{n\, n+1}\cdots R_{12})\cdots (R_{2n-2\, 2n-1}\cdots R_{n-1\, n}\cdot R_{n-1\, n}) (R_{2n-1\, 2n}\cdots R_{n+1\, n+2})\cdot \\
&(R_{n-2\, n-1}\cdots R_{12})\cdots (R_{2n-1\, 2n}\cdots R_{n\, n+1})\\
\stackrel{b}=&(R_{n\, n+1}\cdots R_{12})\cdots (R_{2n-2\, 2n-1}\cdots R_{n-2\, n-1}) (R_{2n-1\, 2n}\cdots R_{n+1\, n+2})\cdot \\
&(R_{n-2\, n-1}\cdots R_{12})\cdots (R_{2n-1\, 2n}\cdots R_{n\, n+1})\quad(\text{as }R_{n-1\, n}^2=\id)\\
\stackrel{(c)}=&\cdots (\text{repeating steps }(a) \text{ and } (b))\\
\stackrel{(d)}=&(R_{n\, n+1}\cdots R_{23})\cdots (R_{2n-2\, 2n-1}\cdots R_{n+1\, n+2})\cdot \\
&(R_{n+1\, n+2}\cdots R_{23})\cdots (R_{2n-1\, 2n}\cdots R_{n\, n+1})\\
=&\cdots (\text{repeating step }(a)\text{ - }\text{step } (d))\\
=&R_{n\, n+1}\cdots R_{2n-1\, 2n}\cdot (R_{2n-1\, 2n}\cdots R_{n\, n+1})\\
=&\id.
\end{align*}

(ii) This can be easily checked from the definition of derived type solutions. Indeed, without loss of generality let us assume that $R(x,y)=(\alpha_x(y), x)$. Then 
\[
R^{n,n}(x_1\cdots x_n, y_1\cdots y_n)
=\big(\alpha_{x_1}\circ\cdots \circ\alpha_{x_n}(y_1)\cdots \alpha_{x_1}\circ\cdots \circ\alpha_{x_n}(y_n), x_1\cdots x_n\big).
\]
This ends the proof.
\end{proof}

Note that one also has 
\[
R^{n,n}
=( R_{n\, n+1}\cdots R_{2n-1\, 2n})(R_{23}\cdots R_{n+1\, n+2})\cdots (R_{12}\cdots R_{n\, n+1}).
\]
If we use this formula, the proof is completely similar. 

\begin{rem}
Making use of the brace theory \cite{CJO14, Rum07}, Bachiller-Ced\'o handle the special case of Theorem \ref{T:cons} in \cite[Theorem 3.4]{BC14}, 
where $R$ is required to be non-degenerate and involutive. 
In this case, by Theorem \ref{T:cons} (i) above, $R^{n,n}$ (which is denoted as $R^{(n)}$ in \cite{BC14}) is also involutive
(note that the proofs of \ref{T:cons} (i) and (ii) do not require that $R$ satisfies the YBE).
Let $R^{n,n}(\bar{x},\bar y)=(f_{\bar x}(\bar y),g_{\bar y}(\bar x))$ for all $\bar x, \bar y\in[N^n]$. Then 
$g_{\bar y}(\bar x)=f_{f_{\bar x}(\bar y)}^{-1}(\bar x)$ by Lemma \ref{L:basic}. So it is suffices to determine $f_{\bar x}(\bar y)$. 

Again, let $R(x,y)=(\alpha_x(y), \beta_y(x))$ for all $x,y\in [N]$. 
By Corollary \ref{C:char}, $R$ is a YBE solution if and only if $\alpha$ can be extended an action of $G_R^+$ on $X$
(and so $\beta$ is a right action of $G_R^+$ on $X$). Heuristically, we 
can work in the semigroup $G_R^+$. 
Notice that if $\bar x=x_1\cdots x_n$ and $\bar y=y_1\cdots y_n$ then 
\[
{\bar x}y=\alpha_{\bar x}(y)\Pi_{i=2}^n \beta_{\alpha_{x_i\cdots x_n}(y)}(x_{i-1}) \beta_y(x_n).
\]
So 
\[
\beta_y(\bar x)=\Pi_{i=2}^n \beta_{\alpha_{x_i\cdots x_n}(y)}(x_{i-1}) \beta_y(x_n).
\]
\[
\alpha_{\beta_{y_1\cdots y_{i-1}}(\bar x)}(y_i)=\alpha_{\beta_{y_{i-1}}\cdots \beta_{y_1}(\bar x)}(y_i)
\]
Then 
\begin{align*}
\bar x \bar y
&=\bar x y_1y_2y_n=\alpha_{\bar x}(y_1)\beta_{y_1}(\bar x)y_2\cdots y_n\\
&=\alpha_{\bar x}(y_1)\alpha_{\beta_{y_1}(\bar x)}(y_2)\beta_{y_2}(\beta_{y_1}(\bar x))y_3\cdots y_n\\
&=\alpha_{\bar x}(y_1)\alpha_{\beta_{y_1}(\bar x)}(y_2)\beta_{y_1y_2}(\bar x)y_3\cdots y_n\\
&=\alpha_{\bar x}(y_1)\alpha_{\beta_{y_1}(\bar x)}(y_2)\alpha_{\beta_{y_1y_2}(\bar x)}(\bar x)\beta_{y_1y_2y_3}(y_4)y_5\cdots y_n\\
&=\alpha_{\bar x}(y_1)\Pi_{i=2}^{n}\alpha_{\beta_{y_1\cdots y_{i-1}}(\bar x)}(y_i)\cdot \beta_{y_1\cdots y_{n}}(\bar x).
\end{align*}
One can check that the element $h_i(\bar x,\bar y)$ in \cite[Theorem 3.14]{BC14} coincides with $\alpha_{\beta_{y_1\cdots y_{i-1}}(\bar x)}(y_i)$:
\[
h_i(\bar x,\bar y)=\alpha_{\beta_{y_1\cdots y_{i-1}}(\bar x)}(y_i).
\]
In fact, 
\begin{align*}
& \text{the above identity holds true} \\
&\Longleftrightarrow \alpha_{h_1(\bar x, \bar y)\cdots h_i(\bar x, \bar y)}^{-1}\alpha_{\bar x y_1\cdots y_{i-1}}=\alpha_{\beta_{y_1\cdots y_{i-1}}(\bar x)}\\
&\Longleftrightarrow \alpha_{\bar x y_1\cdots y_{i-1}}= \alpha_{h_1(\bar x, \bar y)\cdots h_{i-1}(\bar x, \bar y)} \alpha_{\beta_{y_1\cdots y_{i-1}}(\bar x)}\\
&\Longleftarrow {\bar x y_1\cdots y_{i-1}}= {h_1(\bar x, \bar y)\cdots h_{i-1}(\bar x, \bar y)} {\beta_{y_1\cdots y_{i-1}}(\bar x)}.
\end{align*}
So it suffices to show that 
\[
 {\bar x y_1\cdots y_{i-1}}= {h_1(\bar x, \bar y)\cdots h_{i-1}(\bar x, \bar y)} {\beta_{y_1\cdots y_{i-1}}(\bar x)}.
\]
To this end, we use induction. We suppose it is true for $i$, and so 
\[
h_i(\bar x,\bar y)=\alpha_{\beta_{y_1\cdots y_{i-1}}(\bar x)}(y_i).
\] 
Now by our inductive assumption, one has 
\begin{align*}
 {\bar x y_1\cdots y_{i}}
 &= {\bar x y_1\cdots y_{i-1}}y_i\\
 &=  {h_1(\bar x, \bar y)\cdots h_{i-1}(\bar x, \bar y)} {\beta_{y_1\cdots y_{i-1}}(\bar x)}y_i\\
 &=  {h_1(\bar x, \bar y)\cdots h_{i-1}(\bar x, \bar y)} \alpha_{{\beta_{y_1\cdots y_{i-1}}(\bar x)}}(y_i)\beta_{y_{i}\cdots y_1}(\bar x)\\
 &=  {h_1(\bar x, \bar y)\cdots h_{i-1}(\bar x, \bar y)} h_i(\bar x,\bar y)\beta_{y_{1}\cdots y_i}(\bar x).
\end{align*}
We are done. 
\end{rem}

\begin{eg}
Let $R(i,j)=(j+1,i)$ for $1\le i,j\le 2$ with addition modulo 2. Then 
\[
R^{3,3}(i_1 i_2 i_3, j_1j_2j_3)=\big((j_1+1)(j_2+1)(j_3+1), i_1i_2i_3\big)
\]
for all $i_1, i_2, i_3, j_1, j_2, j_3\in [2]$. So $R^{3,3}$ is a YBE solution on $[8]$. 
\end{eg}

\subsection{Periodicity of YBE Solutions}

The notion of periodicity of a $k$-graphs is a little bit involved; it essentially reflects periodicity of its infinite path space 
and symmetry of its representations.  
We will not introduce it here. For the interested reader, refer to \cite{DY091, DY09} for all related information which will be needed in the sequel. 

Let $R$ be a given YBE solution on $[N]$. Recall from Theorem \ref{T:YBkGr} and Remark \ref{R:2YB} (iii) that,
for any $k\ge 2$, there is a single-vertex $k$-graph associated to it.

If $N=1$, then $R$ has to be the trivial solution (which is the same as the identity in this case). It is easy to see that the $k$-graph associated
to $R$ is always periodic for any $k\ge 2$ \cite{DY09}. To avoid this trivial case, we always assume that $n>1$ in this subsection. 

\begin{prop}
\label{P:aper}
Let $R:[N]^2\to [N]^2$ be a solution of the YBE. Then its associated $2$-graph is periodic, if and only if its associated $k$-graph is periodic
for any $k\ge 3$.
\end{prop}

\begin{proof}
Suppose first that the associated $2$-graph of $R$ is periodic. Let $(a,-a)$ $(0\ne a\in \bN)$ be a period. Then $a(\epsilon_i-\epsilon_j)$ ($1\le i<j\le k$)
are periods of the associated $k$-graph of $R$. So its associated $k$-graph is periodic \cite{DY09}. 

Conversely, assume that the associated $k$-graph of $R$ is periodic ($k\ge 3$). Let us say $\sum_{i=1}^ka_i\epsilon_i\in \bZ^k$ be a 
period. So there are two distinct non-zero coefficients $a_{i_0}, a_{j_0}$. Since all $\theta_{ij}$'s are the same in \eqref{E:Fth}, 
$a_{j_0}\epsilon_{i_0}+a_{i_0}\epsilon_{j_0}+\sum_{i\ne i_0, j_0}^ka_i\epsilon_i$ should also be a period. 
Thus  
\[
a_{j_0}\epsilon_{i_0}+a_{i_0}\epsilon_{j_0}+\sum_{i\ne i_0, j_0}^ka_i\epsilon_i-\sum_{i=1}^ka_i\epsilon_i=(a_{j_0}-a_{i_0})(\epsilon_{i_0}-\epsilon_{j_0})
\]
is a period of the associated $k$-graph. This implies $(a_{i_0}, -a_{j_0})$ is a period of the associated $2$-graph.
Thus the associated $2$-graph is periodic. 
\end{proof}

By Proposition \ref{P:aper}, it makes sense to define the periodicity of $R$ as follows. 

\begin{defn}
A solution $R$ of the YBE is said to be \textit{periodic}, if its associated $k$-graph $\bF_R^+$ is periodic for some/all $k\ge 2$, and \textit{aperiodic} otherwise.
\end{defn}

\begin{cor}
\label{C:nondaPe}
A non-degenerate YBE solution is aperiodic.  
\end{cor}

\begin{proof}
Let $R$ a be non-degenerate solution on $[N]$, and suppose that $R(\fs,\ft)=(\alpha_\fs(\ft),\beta_\ft(\fs))$. 
Then, for fixed $\fs\in [N]$, $\alpha_\fs$ and $\beta_\fs$ are bijections on $[N]$. 
Then by \cite[Corollary 4.1]{DY091},  the associated $2$-graph is aperiodic. So, for any $k\ge 2$, the associated $k$-graph is aperiodic by Proposition \ref{P:aper}.
Hence $R$ is aperiodic by definition. 
\end{proof}

But the converse of Corollary \ref{C:nondaPe} is not true. For example, let $R(i,j)=(i,j)$ for $1\le i, j\le 2$, 
and $R(m,3)=(3,m)$ for $1\le m\le 3$. It is the trivial extension of $\id$ on $\{1,2\}$ and $\id$ on $\{3\}$, and so it satisfies YBE. 
Clearly $R$ is degenerate, while the associated $2$-graph to $R$ is aperiodic by \cite[Corollary 4.1]{DY091}. 

Based on \cite{DY091, DY09}, the following characterization is not surprising. 

\begin{cor}
A YBE solution $R$ on $[N]$ is periodic, if and only if $R^{n,n}=\id_{[N^n]^2}$ for some $1\le n\in \bN$. 
\end{cor}

\begin{proof}
From \cite{DY091}, the associated $2$-graph of $R$ is periodic if and only if there is a bijection 
$\gamma$ on $[N^n]$ for some $n\in \bN$ such that 
$R^{n,n}(u,v)=(\gamma(u),\gamma^{-1}(v))$. By Theorem \ref{T:cons}, $R^{n,n}$ is  also a YBE solution on $[N^n]$. 
It follows from Example \ref{Eg:2.7} (ii) that  $\gamma=\id_{[N^n]}$, and so $R^{n,n}=\id_{[N^n]^2}$. 
\end{proof}

\noindent
\textbf{Note:} If $R$ is a YBE solution, its $n$th iteration $R^n$ is not necessarily a YBE solution. Also, there are $R$ which are not  YBE solutions, but $R^n$ 
can be YBE solutions. 
Let us take the following as an example: Let $R(i,j)=(2j,i)$ for $i,j=1,2,3$, where the multiplication is modulo $3$. Then $R^2(i,j)=(2i,2j)$, and $(R^2)^2(i,j)=(i,j)$. It 
is easy to check that both $R$ and $(R^2)^2$ are YBE solutions on $[3]$, 
however $R^2$ is not a YBE solution by Example \ref{Eg:2.7} (ii).

\subsection{Connections with Some Known Results}

The following is \cite[Theorem 3.6]{GM}, which is now an immediate consequence of Theorem \ref{T:YBkGr} and Observation \ref{O:ext}. 

\begin{prop}[\cite{GM}]
\label{P:3.6}
A YBE solution can be uniquely extended a YBE solution on its YB-semigroup.
\end{prop}

\begin{proof}[{\rm\textbf{Sketched Proof}}]
Let $R$ be a YBE solution on $X$. Define a map $\tilde R$ on $G_R^+\times G_R^+$ as follows:
\begin{align*}
\tilde R(x, e_\mt)&=(e_\mt, x),\quad \tilde R(e_\mt, x)=(x,e_\mt) \qforal x\in G_R^+,\\
\tilde R(e_u, e_v)&=(e_{v'},e_{u'}) \qforal u,v\in\bF_X^+,
\end{align*}
where $u',v'\in\bF_X^+$ are uniquely determined by 
\[
R_{12}^{|u|,|v|}(e_u^1, e_v^2)=(e_{v'}^2, e_{u'}^1).
\]

Using the notation in Subsection \ref{SS:newfamily}, for our purpose it suffices to  
show that, for integers $l,m,n\ge 1$, 
\[
\big\langle E_1^l,\bE_2^m,\bE_3^n; \theta_{12}:=R^{l,m}, \theta_{13}:=R^{l,n}, \theta_{23}:=R^{m,n}\big\rangle
\] 
is a $3$-graph. This follows from Observation \ref{O:ext}.  
\end{proof}

\begin{rem}
Two remarks are in order:

(i) Notice that $u',v'\in\bF_N^+$ above are also uniquely determined by 
\[
R_{ij}^{|u|,|v|}(e_u^i, e_v^j)=(e_{v'}^j, e_{u'}^i) \quad (1\le i<j\le 3)
\]
as $R_{12}=R_{13}=R_{23}$. 

(ii) Actually, as claimed in \cite[Theorem 3.9]{GM}, $(\tilde R, G_R^+)$ is a graded braided monoid (see \cite{GM} for definition).  The above proposition already shows that 
$\tilde R$ is a braided relation on $G_R^+$, and the first two identities in Proposition \ref{P:3.6} give the required properties {ML0} and {MR0} there, while 
\begin{itemize}
\item[ ]
{ML1} and {MR2}  correspond to the rule for $R^{|ab|,|u|}_{12}(e_a^1 e_b^1, e_u^2)$,
\item[ ]
{ML2} and {MR1} corresponds to the rule for  $R^{|a|,|uv|}_{12}(e_a^1, e_u^2 e_v^2)$.
\end{itemize}
\end{rem}


Let $R$ be a non-degenerate YBE solution on a set $X$. 
Before going further, let us first recall that another construction of its YB-group $G_R$ described in \cite[Theorem 4]{LYZ}. 

Let $X'$ be another copy of $X$ with $e_\fs'$ denoting the element corresponding to $e_\fs\in X$.
We use $\ol{X}= X\sqcup X'$ to stand for the disjoint union of $X$ and $X'$. 
Let
\[
U(X)=\bigsqcup_{i=0}^\infty \, \ol{X}^i
\]
be the unital free semigroup generated by $\ol{X}$. Let $\sim$ be the equivalence relation on $U(X)$ generated by 
\begin{align*}
& g e_\fs e_\fs' h\sim gh, \ g e_\fs' e_\fs h\sim gh, \\
& ge_\fs e_\ft h\sim g e_j e_i h \quad \text{whenever}\quad R(\fs,\ft)=(j,i),
\end{align*}
where $g,h\in U(X)$, $e_\fs, e_\ft,e_i,e_j\in X$, $e_\fs'\in X'$. 
Then 
\[
G_R=U(X)\slash\sim.
\] 

Repeatedly making use of the graphical interpretation of the braided relation, Lu-Yan-Zhu has shown the following theorem in \cite{LYZ}: 

\begin{prop}[\cite{LYZ}]
\label{P:thm4}
Keep the same notation as above, and let $\iota$ be the canonical map from $X$ to $G_R$. 
Then $R$ can be uniquely extended to a solution $\tilde R$ of YBE on $G_R$ such that  
$\iota \times \iota$ intertwines $R$ and $\tilde R$. 
\end{prop}

\begin{proof}[{\rm\textbf{Sketched Proof}}]
Define $\tilde R: \ol X\to \ol X$ as follows:
\begin{align*}
\tilde {R}(e_{i}',e_{j}')=(e_\ft',e_\fs'),\  \tilde {R}(e_\fs', e_{j})=(e_\ft,e_i'), \  \tilde {R}(e_i, e_\ft')=(e_{j}',e_\fs)
\end{align*}
whenever $R(e_\fs,e_\ft)=(e_{j},e_i)$.

The proof given in \cite{LYZ}  consists of three steps:

\textsf{Step 1:} It is proved that $\tilde R$ satisfies YBE on the set $\ol X$ (\cite[Proposition 5]{LYZ}). 

\textsf{Step 2:} It is shown that $\tilde R$ can be further extended to a YBE solution, still denoted as $\tilde R$, on $U(X)$. 

\textsf{Step 3:} $\tilde R$ on $U(X)$ is shown to be consistent with $\sim$. 

The point we wish to make here is that one can now easily apply Observation \ref{O:ext} to obtain Step 2. Indeed,  
Step 1 shows that $\tilde R$ determines a $3$-graph $\bF_{\tilde R}^+$. 
Then by Observation \ref{O:ext} one can further extend $\tilde R$ to a solution of YBE on $U(X)\setminus\{e_\mt\}$. 
Then define 
$\tilde R(e_\mt, x)=(x,e_\mt)$ and $\tilde R(x,e_\mt)=(e_\mt, x)$ for all $x\in U(X)$
to finish the proof. 
\end{proof}

The non-degenerateness required above is to guarantee that $\tilde R$ is well-defined (cf. \cite[Proposition 3]{LYZ}).  
Probably it should be mentioned that Proposition \ref{P:3.6} can not be applied to prove Proposition \ref{P:thm4},
since, as mentioned in \cite{GM}, generally it can be that $G_R^+$ has no cancellation property.

\subsection{Single-Vertex $k$-Graphs are YB-Semigroups}

\label{SS:kGrYB}

In this subsection, using a very different perspective from above, we prove that every single-vertex $k$-graph is the YB-semigroup of 
a square-free, involutive, (generally) degenerate YBE solution. This could be very useful in future research. 

We do \textit{not}  use our notation assumption \eqref{E:ass} in this subsection. 
Let $k\ge 3$, $1\le N_i\in \bN$ for $1\le i\le k$, and for $1\le i<j\le k$,
$
\theta_{ij}: [N_i]\times [N_j]\to [N_j]\times [N_i]
$
be a bijection. 
As in Section \ref{S:ps}, we write $\theta=\{\theta_{ij}: 1\le i<j\le k\}$, and  $\vartheta=\{\sigma_{j,i}\circ\theta_{ij}: 1\le i<j\le k\}$.

Let $X:=\bigsqcup_{i=1}^k [N_i]$. We now define a bijection $R:X^2\to X^2$ as follows:
\begin{itemize}
\item[(pii)] 
$R|_{[N_i]^2}=\id_{[N_i]^2} \quad (1\le i\le k)$,
\item[(pij)]
$R|_{[N_i]\times [N_j]}= \theta_{ij} \quad (1\le i<j\le k)$,
\item[(pji)]
$R|_{[N_j]\times [N_i]}=\theta_{ij}^{-1}\quad(1\le i<j\le k)$.
\end{itemize}

\begin{thm}
\label{T:equ}
Keep the above notation. Then $\theta$ determines a (single-vertex) $k$-graph $\Fth$, if and only if $R$ is 
a solution of the YBE on $\bigsqcup_{i=1}^k [N_i]$. 

Furthermore, the following properties hold:
\begin{itemize}
\item
 $\Fth\cong G_R^{+}$,
 \item
 $R$ is involutive and square-free,
 \item 
 $R$ is degenerate 
unless $N_i=1$ for all $1\le i\le k$.
\end{itemize}
\end{thm}

\begin{proof} 
For the equivalence, suppose first that $\theta$ determines a $k$-graph. 
Note that property (pii) of $R$ says that $R|_{[N_i]^2}$ is a YBE solution on $[N_i]$.
To verify that $R$ satisfies Eq.~\eqref{E:YBE}, it now suffices to check \eqref{E:YBE} holds true in the following two cases:
 
\textsf{Case (a):} on $[N_i]\times [N_i]\times [N_j]$, 
$[N_j]\times [N_i]\times [N_i]$,  $[N_i]\times [N_j]\times [N_i]$ for $1\le i\ne j\le k$;
 
\textsf{Case (b):}  on $[N_i]\times [N_j]\times [N_\fk]$ for $1\le i\ne j\ne \fk\le k$. 

Case (a) is proved by \cite[Theorem 4.9]{GM} (also cf. Subsection \ref{SS:oldnew} Construction $2^\circ$ (ii)). 
Since $\theta$ determines a $k$-graph,  $\vartheta$ satisfies Eq.~\eqref{E:kYB}.  
Combing this fact with properties (pij), (pji) proves Case (b).

Conversely, if $R$ is a YBE solution on $\bigsqcup_{i=1}^k[N_i]$, then combining \eqref{E:YBE} and properties (pij), (pji) of $R$ yields that
$\vartheta$ satisfies Eq.~\eqref{E:kYB}. Thus $\theta$ determines a $k$-graph. 

For the second part of the theorem, it suffices to show $\Fth\cong G_R^+$, since the other two properties can be easily seen from the definition of $R$. 
To this end, let us write the copy of $[N_i]$ in $G_R^+$ as $\{x_\fs^i: \fs\in [N_i]\}$ ($1\le i\le k$). Thus
\begin{align*}
G_R^+
&=\big\langle {x_\fs^i},  1\le i\le k, \fs\in[N_i]; x^i_\fs x^j_\ft=x^j_{\ft'} x^i_{\fs'} \text{ if } R(x_\fs^i,x_\ft^j)=(x_{\ft'}^j, x_{\fs'}^i),1\le i,j\le k\big\rangle\\
&\stackrel{(\text{pii})}=\big\langle x_\fs^i,  1\le i\le k, \fs\in[N_i]; x^i_\fs x^j_\ft=x^j_{\ft'} x^i_{\fs'} \text{ if } R(x_\fs^i,x_\ft^j)=(x_{\ft'}^j, x_{\fs'}^i), 1\le i<j\le k\big\rangle\\
&=\big\langle x_\fs^i, 1\le i\le k, \fs\in[N_i]; x^i_\fs x^j_\ft=x^j_{\ft'} x^i_{\fs'} \text{ if } \theta_{ij}(\fs,\ft)=(\ft',\fs')\big\rangle
(\text{by }(\text{pij})\, \&\, (\text{pji})) \\
&\cong\big\langle e_\fs^i, 1\le i\le k, \fs\in[N_i]; e^i_\fs e^j_\ft=e^j_{\ft'} e^i_{\fs'} \text{ if } \theta_{ij}(\fs,\ft)=(\ft',\fs')\big\rangle\\
&=\Fth.
\end{align*}
This ends the proof. 
\end{proof}

Some remarks are in order.

\begin{rem}
Thus far, there are two approaches to obtain the relations between $k$-graphs and the YBE -- Theorem \ref{T:YBkGr} and Theorem \ref{T:equ}. 
As mentioned before, the natures of these two approaches are rather different, since they yield two seemingly ``contradictory" inclusions: 
very roughly
\[
\text{Theorem \ref{T:YBkGr}}\Rightarrow \{\text{YBE solutions}\} \subsetneq \{\text{single-vertex }k\text{-graphs with }k\ge 3\},
\]
while 
\[
\text{Theorem \ref{T:equ}}\Rightarrow\{\text{single-vertex }k\text{-graphs with }k\ge 3\}\subsetneq \{\text{YBE solutions}\}.
\]
\end{rem}

\begin{rem}
Keep in mind that it is generally unknown if $G_R^+$ is (even left) cancellative  for a given non-degenerate YBE solution $R$ (cf. \cite{GM}). 
So Theorem \ref{T:equ} is also interesting in the sense that $\Fth$ is always cancellative being a $k$-graph. 
\end{rem}


\begin{rem}
The solution $R$ on $\bigsqcup_{i=1}^k[N_i]$ given in Theorem \ref{T:equ} can be though of being obtained from the iterations of regular extensions
in \cite{GM}. To make precise, for $1\le n\le k$, let $R_n=R|_{(\bigsqcup_{i=1}^n[N_i])^2}$. Then $R_n$ is a solution of the YBE on $\bigsqcup_{i=1}^n[N_i]$
(and so $R_k=R$). Then $R_n$ is a regular extension of solutions $R_{n-1}$ and $\id_{[N_n]^2}$ for all $2\le n\le k$. 
\end{rem}

\begin{rem}
The above proposition is for $k$-graphs with $k\ge 3$. Note that the case of $k=2$ is nothing but Construction $2^\circ$ (ii) in Subsection \ref{SS:oldnew}.   
\end{rem}

Two examples related to Theorem \ref{T:equ} are given below. 

\begin{eg}
Let $\theta_{13}(\fs,\ft)=\theta_{23}(\fs,\ft)=((\fs+\ft) \text{\ mod 2}, \ft)$ for all $\fs,\ft\in[2]$, and $\theta_{12}=\id$ on $[2]$. 
It is easy to check that $\theta=\{\theta_{12}, \theta_{13}, \theta_{23}\}$ determines a $3$-graph (also cf. \cite[Example 2.3]{DY09} but with
notation difference caution). 
 Then the YBE solution $R$ on $\bigsqcup_{i=1}^3[2]$ determined by $\theta$ in Theorem \ref{T:equ} is given by 
\begin{align*}
R(x_\fs^i,x_\ft^i)&=(x_\fs^i,x_\ft^i), \  R(x_\fs^1,x_\ft^2)=(x_\fs^2,x_\ft^1),\\
R(x_\fs^1,x_\ft^3)&=(x_{\fs+\ft}^3,x_\ft^1),\
R(x_\fs^2,x_\ft^3)=(x_{\fs+\ft}^3,x_\ft^2)  
\end{align*}
for all $1\le i\le 3$ and $\fs,\ft\in[2]$. 
\end{eg}

\begin{eg}
Let $\theta_{ij}=\id_{[2]^2}$ for all $1\le i<j\le 3$.
 Using the notation in the above example, then the YBE solution $R$ on $\bigsqcup_{i=1}^3[2]$ determined by $\theta$ in Theorem \ref{T:equ} is given by 
\begin{align*}
R(x_\fs^i,x_\ft^i)&=(x_\fs^i,x_\ft^i)\quad (1\le i\le 3),\\
  R(x_\fs^i,x_\ft^j)&=(x_\fs^j,x_\ft^i)\quad (1\le i<j\le 3).
\end{align*}
Notice that $R$ is not the identity solution any more. 
\end{eg}

\section{Product conjugacy and YB-homomorphisms}
\label{S:pc}

In this very short section, we briefly discuss product conjugacy and YB-homomorphisms of solutions of the YBE. 
The former gives an invariant of their associated $k$-graphs, while the latter yields a length-preserving invariant of their 
YB-semigroups. 

\begin{defn}
Let $\theta$ and $R$ be bijections on $[N]\times [N]$.
Say that $\theta$ and $R$ are \textit{product conjugate} if there are $\tau, \rho\in \sym_{[N]}$ 
such that $R\circ(\tau\times \rho)=(\tau\times \rho)\circ \theta$. 
\end{defn}

Suppose that $\theta$ and $R$ are product conjugate. Then it is easy to check that 
their associated single-vertex $2$-graphs $\Fth$ and $\bF_R^+$ are isomorphic \textit{semigroups}.
But in general, this is not an invariant (one could switch the two colours)\cite{Pow07}. However, if one requires that the isomorphism is degree-preserving, 
then \cite[Proposition 2.11]{FS01} gives the converse. 

Before giving the statement, note that ,for $k\ge 3$, a $k$-graphs $\Fth$ is a symmetric tensor groupoids in the sense of \cite{FS01}: Indeed, 
$\sigma_{X, Y}: X\otimes Y\to Y\otimes X, (x,y)\mapsto (y,x)$ is the required flip functor. Note that $xy$ and $yx$ always make sense
because of $\Fth$ being single-vertex.

\begin{prop}[\cite{FS01}]
\label{P:2.11}
Two solutions $\theta, R:  [N]^2\to [N]^2$ of the YBE are product conjugate, if and only if,
for any $k\ge 2$, 
their associated $k$-graphs $(\bF_{\theta}^+,d)$ and $(\bF_R^+,d)$ are (graph) isomorphic. 
\end{prop}

\begin{defn}
Let $R_X$ and $R_Y$ be solutions of the YBE on $X$ and $Y$, respectively. We say that $R_X$ and $R_Y$ are \textit{YB-homomorphic}, 
if there is a map $\phi: X\to Y$ such that 
\[
(\phi\times \phi)\circ R_X=R_Y\circ(\phi\times\phi).
\]  
If the homomorphism $\phi$ is bijective, then $R_X$ and $R_Y$ are called \textit{YB-isomorphic}, and $\phi$ is a called 
a \textit{YB-isomorphism} from $R_X$ to $R_Y$. 
\end{defn}

Clearly, $R_X$ and $R_Y$ being YB-isomorphic is stronger than $R_X$ and $R_Y$ being product conjugate. 

Let $R$ be a YBE solution on $X$. From the definition of its YB-semigroup, it is easy to see that 
$G_R^+=\bigsqcup_{n=0}^\infty X_n$, where $X_n$ is the set of all elements in $G_R^+$ of length $n$. 
But applying Proposition \ref{P:2.11} and the relations between $\bF_R^+$ and $G_R^+$, one has 

\begin{cor}
Two YBE solution are YB-isomorphic, if and only if their YB-semigroups are length-preserving isomorphic. 
\end{cor}

In fact, if $\pi: G_{R_X}^+\to G_{R_Y}^+$ is an isomorphism preserving length, then its restriction $\pi|_X$ onto the generator set $X$
gives a YB-isomorphism between $R_X$ and $R_Y$. 


\begin{rem}
The above results can \textit{not} be generalized to YB-groups, since the canonical mapping from $X$ to $G_X$ may not be an embedding. 
If $R_1$ and $R_2$ are YB-isomorphic solutions, then one can easily check that $G_{R_1}$ and $G_{R_2}$ are isomorphic. But the converse is not true. 
Let us look at the following example. 

Let $R_1(i,j)=(j,i)$ for $1\le i,j\le 2$. Then $R_1$ is a YBE solution on $[2]$.
Let $R_2(i,j)=((j+1)\, {\rm mod} 2,i)$ for $1\le i,j\le 2$ and $R_2(i,3)=(3,i)$ for $1\le i\le 3$. Then $R_2$ is a YBE solution on $[3]$ being the trivial extension.
Clearly $G_{R_1}\cong G_{R_2}\cong \bZ^2$. 
However, it is trivial to see that $R_1$ and $R_2$ are not isomorphic. 
\end{rem}

\section{Non-degeneracy $=$ unique pullback and pushout properties}
\label{S:nondeg}

As seen from above, a solution of the YBE and its associated $k$-graph heavily interact with each other. We in this section show that 
the non-degeneracy of a YBE solution is characterized by the unique pullback and pushout properties of its associated $k$-graphs. 
As a result, we show that a YBE solution is non-degenerate if and only if so are all of its higher levels. 

\begin{lem}
\label{L:nondeg}
Let $R(x,y)=(\alpha_x(y),\beta_y(x))$ be a bijection on $X^2$. Then $R$ is non-degenerate, if and only if 
the following properties hold: 

\begin{itemize}
\item[(i)] for any $(x,y')$ there is a unique $(y,x')$ such that $R(x,y)=(y',x')$; 

\item[(ii)] for any $(y,x')$, there is a unique $(x,y')$ such that $R(x,y)=(y',x')$. 
\end{itemize}
\end{lem}

\begin{proof}
(i) The ``Only if" part is \cite[Lemma 4]{LYZ}. Actually, the uniqueness of $y$ follows from the bijectivity of $y\mapsto y'$ for fixed $x$. Then the uniqueness of $x'$ 
comes from the bijectivity of $R$. 

For the ``If" part, let us fix $x\in X$. The surjectivity and injectivity of the map $y\mapsto y'$ is from the give existence and uniqueness, respectively.

(ii) This is proved similarly.
\end{proof}

Pictorially,  one can draw $R:X\times X\to X\times X$ as a crossing, which is called an \textit{$R$-square}:
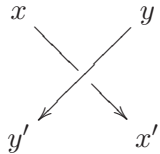
\[
\input xy
\xyoption{all} \xymatrix@!@C=.1in@R=.1in{
&x \ar@{-}[dr]&&y\ar@{->}[ddll]\\
&& \ar@{->}[dr]&\\
&y'&& x'
}
\]
It follows from Lemma \ref{L:nondeg} that any row/column of the $R$-square completely determines it if $R$ is bijective/non-degenerate.

Motivated by the notions of the little pullback and pushout properties introduced in \cite{Exe}, let us give the following definition.

\begin{defn} 
\label{D:pu}
Let $\Fth=\big\langle e_\fs^i; \text{($\theta$-CR)}\big\rangle$ be an arbitrary single-vertex $k$-graph. We say that 
\begin{itemize}
\item[(a)]
$\Fth$ has the \textit{unique pullback property} if for any $(e^i_\fs, e^j_\ft)$ $(i\ne j$) there is a unique pair
$(e^j_{\ft'}, e^i_{\fs'})$ such that $e_\fs^i e_{\ft'}^j=e_{\ft}^j e_{\fs'}^i$; and 
\item[(b)]
$\Fth$ has the \textit{unique pushout property} if for any $(e^i_{\fs'}, e^j_{\ft'})$ $(i\ne j$) there is a unique pair
$(e^j_{\ft}, e^i_{\fs})$ such that $e_{\fs}^i e_{\ft'}^j=e_{\ft}^j e_{\fs'}^i$. 
\end{itemize}
\end{defn}

Graphically, the unique pullback (resp. pushout) property of $\Fth$ says that the lower (resp. upper) part determines the upper (resp. lower) part in the following commuting
diamond:
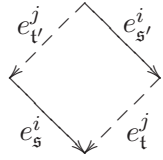
\[
\begin{xy}
{\ar (0, 0); (10,-10)};
?*!/_3mm/{e_{\fs'}^i};
{\ar@{-->} (10, -10); (0,-20)};
?*!/_3mm/{e_{\ft}^j};
{\ar@{-->}(0, 0); (-10,-10)};
?*!/^3mm/{e^j_{\ft'}};
{\ar (-10,-10); (0,-20)};
?*!/^3mm/{e_{\fs}^i};
\end{xy}
\]
\\
Therefore, combining the factorization property, one has that any two (not necessarily directed) adjacent edges completely determine the whole commuting diamond.  

Note that the difference between the little pullback property defined in \cite{Exe} and the unique pullback property
in Definition \ref{D:pu} lies in that the former requires at most one such pair $(e^j_{\ft'}, e^i_{\fs'})$, but not require the existence. 
Likewise for the unique pushout property. 

The following lemma implies that the unique pullback and pushout properties can be generalized to some `higher levels' of $k$-graphs. 

In the sequel, for $\bm,\bn\in \bN^k$, as usual, the symbol $\bm\wedge \bn$ denotes the coordinate-wise minimum of $\bm$ and $\bn$. 

\begin{lem}
\label{L:pullpush}
Let $\Fth$ be an arbitrary single-vertex $k$-graph, and 
$\mu,\nu\in\Fth$ be such that $d(\mu)\wedge d(\nu)=\mathbf{0}$.
\begin{enumerate}
\item Suppose that $\Fth$ has the unique pullback property. Then 
there are unique $\tilde\mu,\tilde\nu\in\Fth$ with $d(\mu)=d(\tilde\mu)$ and $d(\nu)=d(\tilde\nu)$ such that $\mu\tilde\nu=\nu\tilde\mu$.

\item Suppose that $\Fth$ has the unique pushout property. Then 
there are unique $\tilde\mu,\tilde\nu\in\Fth$ with $d(\mu)=d(\tilde\mu)$ and $d(\nu)=d(\tilde\nu)$ such that $\tilde\mu\nu=\tilde\nu\mu$.   
\end{enumerate}
\end{lem}

\begin{proof} 
For $\bn=(n_1,\ldots, n_k)\in\bN^k$, let us write $|\bn|=n_1+\cdots +n_k$. 

(i) It is proved by the higher dimensional version of mathematical induction.  
If $d(\mu)=\ep_i$ and $d(\nu)=\ep_j$ for some $i\ne j$, then (i) holds true by the definition of the unique pullback property. 
Assume that (i) holds true for $\mu$, $\nu\in\Fth$ such that 
$|d(\mu)+d(\nu)|\le N$ and $d(\mu)\wedge d(\nu)=\mathbf{0}$.  
We wish to show that (i) is true if $\mu,\nu\in\Fth$ are such that $|d(\mu)+d(\nu)|=N+1$ and $d(\mu)\wedge d(\nu)=\mathbf{0}$. 

Let us write $\nu=\nu_2\nu_1$ satisfying $d(\nu_1), d(\nu_2)>\mathbf{0}$. One then has $d(\mu)\wedge d(\nu_i)=\mathbf{0}$ for $i=1,2$. 
By our inductive assumption, there are unique elements $\mu'$ and $\nu_2'$ such that 
 $d(\mu')=d(\mu)$,  $d(\nu_2')=d(\nu_2)$ and $\mu \nu_2'=\nu_2 \mu'$. Applying our inductive assumption again, 
 one has unique elements $\mu''$ and $\nu_1'$ such that 
 $d(\mu'')=d(\mu)$,  $d(\nu_1')=d(\nu_1)$ and $\nu_1\mu'' =\mu' \nu_1'$. 
 Hence the pair $(\tilde\mu,\tilde\nu):=(\mu'', \nu_2'\nu_1')$ satisfies
 $\mu\tilde\nu=\nu\tilde\mu$ with $d(\tilde\mu)=d(\mu)$ and $d(\tilde\nu)=d(\nu)$. 
 This finishes the proof of existence.
The above proof can be easily summarized in the following picture:  
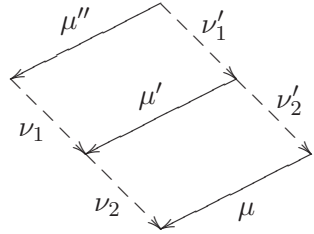
\[
\begin{xy}
{\ar@{-->} (-20, 0); (-10,-10)};
?*!/^3mm/{\nu_1};
{\ar@{-->} (0, 10); (10,0)};
?*!/_3mm/{\nu_1'};
{\ar@{-->} (10, 0); (20,-10)};
?*!/_3mm/{\nu_2'};
{\ar (0, 10); (-20,0)};
?*!/^3mm/{\mu''};
{\ar (20, -10); (0,-20)};
?*!/_3mm/{\mu};
{\ar(10, 0); (-10,-10)};
?*!/^3mm/{\mu'};
{\ar@{-->} (-10,-10); (0,-20)};
?*!/^3mm/{\nu_2};
\end{xy}
\]
  
It remains to show uniqueness. To this end, let us assume that $(\tilde \mu'', \tilde\nu')$ also satisfies 
the required properties. Write $\tilde\nu'= \tilde\nu_2'\tilde\nu_1'$ such that  $d(\nu_2')=d(\tilde\nu_2')$ and 
$d(\nu_1')=d(\tilde\nu_1')$. So we have $\nu_2\nu_1\tilde\mu''=\mu\tilde\nu_2'\tilde\nu_1'$. This implies 
$\nu_2\tilde\mu'\tilde\nu_1=\mu\tilde\nu_2'\tilde\nu_1'$ using $\nu_1\tilde\mu''=\tilde\mu'\tilde\nu_1$. 
By the unique pullback property, $(\mu,\nu_2)$ gives $(\tilde\mu',\tilde\nu_2')=(\mu',\nu_2')$. 
Then applying the unique pullback property to $(\mu',\nu_1)$ gives
$(\nu_1', \mu'')=(\nu_1',\tilde\mu'')$. 
 Thus $(\tilde\mu'', \tilde\nu')=(\tilde\mu'', \tilde\nu_2'\tilde\nu_1')=(\mu'',\nu_2'\nu_1')=(\mu'',\nu')$, as desired.
 
 (ii) It is proved similarly. 
\end{proof}

The following result follows immediately from Definition \ref{D:pu} and Lemma \ref{L:nondeg}.  

\begin{cor}
\label{C:nondeg}
Let  $R: X^2\to X^2$ be a bijection. 
\begin{enumerate}
\item
$R$ is non-degenerate, if and only if the associated $2$-graph $\bF_R^+$ has the unique pullback and pushout properties. 

\item
Suppose that $R$ is a YBE solution. Then it is non-degenerate, if and only if its associated $k$-graph $\bF_R^+$ has the unique pullback and pushout properties
for any $k\ge 2$. 
\end{enumerate}
\end{cor}

\begin{prop}
Suppose that $R:X^2\to X^2$ is a bijection. 
\begin{enumerate}
\item
$R$ is non-degenerate, if and only if so is its $n$th level $R^{n,n}$ for every $n\ge 1$. 

\item 
$R$ is a non-degenerate YBE solution on $X$, if and only if $R^{n,n}$ is a non-degenerate YBE solution on $X^n$ for every $n\ge 1$.
\end{enumerate}
\end{prop}

\begin{proof}
(i) Suppose that $R$ is non-degenerate. By Corollary \ref{C:nondeg} (i), its associated $2$-graph $\bF_R^+$ has the unique pullback and pushout properties. 
It follows from Lemma \ref{L:pullpush} that the $n$th level $2$-graph $\bF_{R^{n,n}}^+$ has the unique pullback and pushout properties.  
Applying Corollary \ref{C:nondeg} (i) again gives the non-degeneracy of $R^{n,n}$. 

(ii)  This is straightforward from (i) and Theorem \ref{T:cons}. 
\end{proof}

\section{Homology and cohomology Theories for the YBE}

\label{S:hom}

In this last section, we investigate the homology and cohomology theories for the YBE. 
The homology theory for the YBE was first introduced and studied in \cite{CES04}, 
which has important applications to link invariants in knot theory \cite{CJKLS03, CES04}. 
We here give two notions of the homology for the YBE -- the YB-homology and the semigroup homology of the YBE. 
Likewise for the cohomology. Both notions are motivated from the interplay between the YBE and $k$-graphs studied above, and 
the ideas of Kumjian-Pask-Sims in \cite{KPS12, KPS}.

Our purpose of doing so is twofold: First of all, it makes the arguments in \cite{CES04} more transparent in a certain sense. 
For instances, the uniqueness result on colouring \cite[Lemma 2.3]{CES04}, which is vital in defining the homology of the YBE in \cite{CES04}, 
now can be easily proved from our viewpoint. Secondly, we hope that our point of view would be useful in further studies of the homology and cohomology
theories for the YBE.

\subsection{YB-Homology}

Let $R$ be a YBE solution on a set $X$.

\begin{defn}
\label{D:HomYBE}
Set $C_n(R)=\bZ X^n$. Define $\partial_n:C_n(R)\to C_{n-1}(R)$ by 
\begin{align*}
\partial_n(x_1,\ldots,x_n)=\sum_{i=1}^n(-1)^i\big(F_i^0(x_1,\ldots,x_n)-F_i^1(x_1,\ldots,x_n)  \big),
\end{align*}
where 
\begin{align*}
F_i^0(x_1,\ldots,x_n)&=R_{n-1\, n}\cdots R_{i,i+1}(x_1,\ldots,x_n) \ominus \text{the }n\text{th} \text{ component},\\
F_i^1(x_1,\ldots,x_n)&=R_{12}\cdots R_{i-1,i}(x_1,\ldots,x_n) \ominus \text{the }1\text{st} \text{ component}.
\end{align*}
Here $R_{ij}$ is the leg notation of $R$ on $X^n$, and $\ominus$ means ``deleting". 
\end{defn}

\begin{rem}
\label{R:hom}
It is useful to think of the operations $F_i^0$ and $F_i^1$ as follows.
The operation $F_i^0$ moves ``$x_i$"  from $(x_1,\ldots,x_n)$ to its right end by $R$  while also preserving the order of other elements,
and then deletes ``$x_i$":
\begin{align*}
(x_1,\ldots, x_n)\stackrel{R_{n-1\, n} \cdots R_{i,i+1}}\longrightarrow &(x_1,\ldots,x_{i-1},x'_{i+1},\ldots,x'_n,x'_i)\\
\stackrel{\text{delete }e'_{\fs_i}}\longrightarrow \  \ & (x_1,\ldots,x_{i-1},x'_{i+1},\ldots,x'_n)\\
=\quad\ &F_i^0(x_1,\ldots,x_n).
\end{align*}

The operation $F_i^1$ moves ``$x_i$"  from $(x_1,\ldots,x_n)$ to its left end by $R$ while also preserving the order of other elements,
and then deletes ``$x_i$":
\begin{align*}
(x_1,\ldots, x_n)\stackrel{R_{12} \cdots R_{i-1,i}}\longrightarrow &(x'_i,x'_1,\ldots,x'_{i-1},x_{i+1}\ldots,x_n)\\
\stackrel{\text{delete }x'_{i}}\longrightarrow \ & (x_1',\ldots,x_{i-1}',x_{i+1}\ldots,x_n)\\
=\quad &F_i^1(x_1,\ldots,x_n).
\end{align*}
\end{rem}

\begin{rem}
\label{R:2}
As we already mentioned above, the idea here is motivated from \cite{KPS12, KPS} by Kumjian-Pask-Sims.
As a matter of fact, let $\lambda=(x_1,\ldots,x_n)$. We can think the operations $F_i^0(\lambda)$ and $F_i^1(\lambda)$ in view of the associated $n$-graph $\bF_R^+$ to 
$R$ (cf. Theorem \ref{T:YBkGr} (iii))\footnote{Actually, one can obtain $F_i^0(\lambda)$ and $F_i^1(\lambda)$ from any $k$-graph ($k\ge n$) associated to $R$.}:
regard $x_i$ as an edge of degree of $\epsilon_i$, and 
so $\lambda$ is a path of degree $d(\lambda)=\epsilon_1+\cdots+\epsilon_n$ in $\Lambda$. Then $F^0_i(\lambda)$ and $F^1_i(\lambda)$ are 
unique paths (with the increasing subscript order) in $\bF_R^+$ such that  
\[
F_i^0(\lambda)\mu_i=\lambda=\nu_iF^1_i(\lambda)
\]
with $d(\mu_i)=d(\nu_i)=\epsilon_i$. 
So we really connect the YB-homology of $R$ with the homology of its ``$\infty$-graph" in the sense of \cite{KPS12, KPS}. 
\end{rem}

\begin{rem}
It is known from \cite{KPS} that, for the homology of a $k$-graph $\Lambda$, its $n$th boundary map is always trivial for every $n>k$. 
But this is not the case any more for the YB-homology from Remark \ref{R:2} above. 
\end{rem}

As in \cite{KPS12, KPS}, it is not hard to check that $(X_*,\partial_*)$ is a chain complex.  
Define the \textit{$n$th YB-homology group} of $(X,R)$ by 
\[
H_n^\text{YB}(R)=\ker\partial_n/\im\partial_{n+1}.
\]
We call $H_*^\text{YB}(R)$ the \textit{YB-homology} of $(X,R)$. 

For an abelian group $A$, let $C^n(X, A)$ be the set of homomorphisms from $C_n(R)$ to $A$. Define $\delta^n:C^n(X, A) \to C^{n+1}(X, A)$ by
\[
(\delta^n f)(x_1,\ldots, x_{n+1})=f\circ\partial_{n+1}(x_1,\ldots, x_{n+1}).
\]
Define the \textit{$n$th cohomology group} by 
\[
H^n_{\text{YB}}(R,A)=\ker\delta^n/\im\delta^{n-1}
\] 
and call $H^*_{\text{YB}}$ the \textit{YB-cohomology} of $(X,R)$.

\subsection{Semigroup Homology of the YBE}

Let $R$ be a YBE solution on $X$, and $G_R^+$ be its YB-semigroup. The \textit{semigroup homology and cohomology of $R$},
by definition, are the homology and cohomology of $G_R^+$, respectively: 
\begin{align*}
\ul{H}_n^\text{YB}(R)&:=H_n(G_R^+),\\
\ul{H}^n_{\text{YB}}(R,A)&:=H^n(G_R^+,A)
\end{align*}
for any ablian group $A$. 

%
%

Motivated by \cite{KPS12, KPS}, it would be very interesting to know the answer to the following question.

\begin{ques}
Do the YB-cohomology and semigroup cohomology of the YBE coincide?
\end{ques}

If the answer to the above question is affirmative, then $H^*_{\text{YB}}(R,A)=H^*(G_R^+,A)$.
So a non-classical topic becomes a very classical one. 
Furthermore, if $G_R$ is a group of fractions for $G_R^+$, then the situation is even better: $H^*_{\text{YB}}(R,A)=H^*(G_R,A)$.
In this case, we will have a lot of tools to apply in our hands. 


\subsection{An Applications to Derived-Type YBE Solutions}

For derived-type YBE solutions, $\partial_n$ (and so $\delta_n$) has particularly nice form. 
The YB-homology and YB-cohomology are closely related to the quandle homology and cohomology, respectively. 
In fact, a quandle is a (non-empty) set $X$ with a binary relation $*$ such that 
$
R(x,y)=(y, x*y)) 
$ 
($x,y\in X$)
is a square-free YBE solution on $X$, which is automatically of derived-type. 
Thus all known quandle homology and cohomology results, e.g., in \cite{CES04, CJKLS03}, can be applied to this special class of the YBE solutions.

Let 
\[
R(x,y)=(y,\beta_y(x))=:(y, x*y)
\] 
be a derived type YBE solution on $X$. So in this case one has
\begin{align*}
x_1\cdots x_n
&=x_1\cdots x_{i-1}x_{i+1}\ldots x_n\big((x_i*x_{i+1})*\cdots *x_n\big)\\
&=x_i(x_1*x_i)\cdots (x_{i-1}*x_i)x_{i+1}\cdots x_n.
\end{align*}
So 
\begin{align*}
&\partial_n(x_1\cdots x_n)\\
&=\sum_{i=2}^{n}(-1)^i\big(x_1\cdots x_{i-1}x_{i+1}\cdots x_n-(x_1*x_i)\cdots (x_{i-1}*x_i)x_{i+1}\cdots x_n\big).
\end{align*}
For example, it is easy to compute 
\begin{align*}
\partial_1(x)&=0,\\
\partial_2(x_1, x_2)&=x_1-x_1*x_2,\\
\partial_3(x_1,x_2,x_3)&=(x_1,x_3)-(x_1*x_2,x_3)-(x_1,x_2)+(x_1*x_3,x_2*x_3).
\end{align*}
Hence any 1-cocycle $\phi:X\to A$ is of the form 
\[
\phi(x_1)=\phi(x_1*x_2)\qforal x_1,x_2\in X,
\]
and any $2$-cocycle $\phi:X\times X\to A$ is of the form
\[
\phi(x_1,x_2)+\phi\big(x_1*x_2, x_3\big)=\phi(x_1,x_3)+\phi\big(x_1*x_3, x_2*x_3\big).
\]

If we let $O_\beta=\{\beta_y(x): x,y \in X\}$ be the orbit of the right action $\beta$, then the following consequence is immediate: 
$H_1^{YB}(R)$ is isomorphic to the orbit class of $\beta$, and $H^1_{YB}(R)$ is the set of functions $\phi: X\to A$ which are $\beta$-invariant.


\end{document}